\documentclass[]{interact}

\usepackage{epstopdf}

\usepackage{
	amssymb,
	amsmath,
    caption,
    subcaption
}
\usepackage{subcaption}
\usepackage{graphicx}
\usepackage{xcolor}

\counterwithin{figure}{section}

\usepackage[numbers,sort&compress]{natbib}
\bibpunct[, ]{[}{]}{,}{n}{,}{,}

\usepackage{hyperref}

\theoremstyle{plain}
\newtheorem{theorem}{Theorem}[section]
\newtheorem{lemma}[theorem]{Lemma}
\newtheorem{corollary}[theorem]{Corollary}
\newtheorem{proposition}[theorem]{Proposition}

\theoremstyle{definition}
\newtheorem{definition}[theorem]{Definition}

\theoremstyle{remark}
\newtheorem{remark}[theorem]{Remark}

\def\R{\mathbb R}

\def\N{\mathbb N}
\def\I{\mathbb I}

\def\H{\mathcal H}
\def\G{\mathcal G}
\def\M{\mathcal M}

\def\C{\mathcal C}

\def\tto{\rightrightarrows}

\def\prox{\text{prox}}
\def\re{\text{Re}}
\def\zer{\text{zer }}
\def\fix{\text{Fix }}
\def\weakto{\rightharpoonup}

\begin{document}

\articletype{ARTICLE}

\title{The Degenerate Variable Metric Proximal Point Algorithm and Adaptive Stepsizes for Primal-Dual Douglas-Rachford}

\author{
    \name{
        Dirk A. Lorenz\textsuperscript{a}, Jannis Marquardt\textsuperscript{b} and Emanuele Naldi\textsuperscript{a}
        \thanks{
            \textsuperscript{a}Institute for Analysis and Algebra, TU Braunschweig, Germany (d.lorenz@tu-braunschweig.de, e.naldi@tu-braunschweig.de).
        }
        \thanks{
            \textsuperscript{b}Institute for Partial Differential Equations, TU Braunschweig, Germany (j.marquardt@tu-braunschweig.de).
        }
    }
}

\maketitle

\begin{abstract}
In this paper the degenerate preconditioned proximal point algorithm will be combined with the idea of varying preconditioners leading to the degenerate variable metric proximal point algorithm. The weak convergence of the resulting iteration will be proven. From the perspective of the  degenerate variable metric proximal point algorithm, a version of the primal-dual Douglas-Rachford method with varying preconditioners will be derived and a proof of its weak convergence which is based on the previous results for the proximal point algorithm, is provided, too. After that, we derive a heuristic on how to choose those varying preconditioners in order to increase the convergence speed of the method.
\end{abstract}

\begin{keywords}
Preconditioned proximal point algorithm, varying preconditioners,  Douglas–Rachford method, non-stationary primal-dual method, adaptive stepsizes
\end{keywords}

\begin{amscode}
47H05, 65K05, 90C25
\end{amscode}

\section{Introduction}
The performance of first order splitting methods for monotone inclusions often depends critically on stepsize choices, i.e. they perform well for a narrow range of stepsizes, but convergence (although sometimes guaranteed for all positive stepsizes) can slow down considerably for other choices (see, e.g.~\cite{adaptive-stepsizes}). Hence, we are interested in adaptive stepsize choice that may have the ability to automatically find good stepsizes.
In this work we consider specifically the primal-dual Douglas-Rachford (DR) method~\cite{oconnor} and develop an adaptive stepsize. To do so (and also, to prove convergence of our stepsize heuristic) we consider the method as a degenerate preconditioned proximal point iteration~\cite{degenerate-pppa}. This leads to the degenerate variable metric proximal point algorithm and we prove weak convergence for this method. 

We describe the setup of this paper in more detail: Let $\H$ be a real Hilbert space and $T:\H\tto\H$ a (possibly set valued) maximal monotone operator. Formally, a set valued operator $T$ is a map from $\H$ to the power set $2^\H$ and is completely described by its graph $\G[T]$ by the relation $y\in Tx \Leftrightarrow (x,y)\in\G[T]$. A solution of the inclusion problem
\begin{equation}\label{pp-inclusion}
    \text{find }u\in\H \text{ such that } 0\in Tu.
\end{equation}
 is called zero of $T$ and we write $u\in\zer T$. The proximal point algorithm aims to find such a zero by iterating the resolvent $J_T := (\I + T)^{-1}$ of $T$, where $\I$ denotes the identity mapping. Since $T$ is assumed to be maximal monotone, $J_T$ is a full domain and single valued mapping by Minty's surjectivity theorem \cite{minty}. Furthermore, $J_T$ is firmly non-expansive and its fixed points correspond to the zeros of \eqref{pp-inclusion}, i.e. $\fix J_T = \zer T$ (cf. \cite[Section 23]{bauschke}). For every starting point $u^0\in \H$ the sequence defined by the recurrence $u^{k+1} = J_Tu^k$ weakly converges to a solution of \eqref{pp-inclusion} (cf. \cite[Theorem 23.41]{bauschke}).

The computation of $J_T$ is in general an expensive task. In certain situations, this can be changed by using preconditioning. For a linear, self-adjoint and positive-definite map $\M:\H\to\H$, the replacement of the inclusion $0\in Tu$ by $0\in\M^{-1}Tu$ (where $\M^{-1}T$ is the set valued operator which is characterized by $v\in \M^{-1}Tu$ exactly if $\M v\in Tu$) results in the iteration
\[
  u^0\in\H, \quad u^{k+1} = J_{\M^{-1}T}u^k = (\I + \M^{-1}T)^{-1}u^k = (\M+T)^{-1}\M u^k\]

and problem \eqref{pp-inclusion} is equivalent to
\begin{equation}
    \text{find }u\in\H \text{ such that } 0\in (\M+T)^{-1}\M u
\end{equation}
The convergence analysis of this preconditioned proximal point iteration for such preconditioner $\M$ can be accomplished as for the unconditioned iteration after exchanging the inner product $\langle u,v\rangle$ in $\H$ with $\langle u,v\rangle_\M := \langle \M u, v\rangle$. This changes if $\M$ is not positive definite, but only positive semi-definite as has been proposed in~\cite{degenerate-pppa}. For such degenerate preconditioners, $\langle u,v\rangle_\M$ is not necessarily an inner product and $\M^{-1}$ or $(\M+T)^{-1}\M$ may not be single valued. The notion of admissible preconditioners comes into play.

\begin{definition}[Admissible preconditioner]
    A bounded, linear, self-adjoint and positive semi-definite operator $\M: \H\to\H$ is called an admissible preconditioner for $T:\H\tto\H$ if 
    \begin{align*}
    J_T^\M := (\M + T)^{-1}\M
    \end{align*}
    is single valued and has full domain.
\end{definition}

The operator $J_T^\M$ may be interpreted as an instance of the so called warped resolvents (cf. \cite{warped-resolvents}) and in~\cite{degenerate-pppa} it has been shown that the iteration $u^{k+1} = J_{T}^{\M}u^{k}$ does converge weakly to a zero of $T$ if $(\M+T)^{-1}$ is Lipschitz continuous.
In this work we propose a non-stationary variant in which the preconditioner $M$ varies with $k$ (also called \emph{variable metric} method in this context~\cite{burke1999variable-metric-pp}), i.e. we consider a sequence of admissible preconditioners $(\M_k)_{k\in\N}$ which results in the degenerate variable metric proximal point algorithm
\begin{equation}\label{vary-precond-iter}
    u^0\in\H,\quad u^{k+1} = J_T^{\M_k}u^k.
\end{equation}

Degenerate preconditioning is especially effective for splitted inclusion problems as
\begin{equation}\label{dr-inclusion}
    \text{find }x\in\H \text{ such that } 0\in (A+B)x
\end{equation}
for two maximal monotone operators $A,B:\H\tto \H$, and the celebrated DR method (cf.~\cite{lions-mercier, dr-eckstein})
\begin{equation}\label{dr-iteration-unconditioned}
  w^0\in\H,\quad w^{k+1}= w^k + J_{ B}(2J_{ A}w^k - w^k) - J_{A}w^k
\end{equation}
can be seen as an instance of this (see~\cite{procond-dr,degenerate-pppa}). Notice that, if~\eqref{dr-iteration-unconditioned} converges, the iterates $w^{k}$ do not converge to a solution of $0\in(A+B)x$, but the sequence $(J_{A} w^k)_{k\in\N}$ does. The given iteration may be derived from the perspective of the preconditioned proximal point iteration (cf. \cite{dr-eckstein}) and \eqref{dr-iteration-unconditioned} may be seen as a special case of \eqref{vary-precond-iter} for a certain choice of $T$ and $(\M_k)_{k\in\N}$~\cite{degenerate-pppa}.

However, we can as well apply preconditioning to~\eqref{dr-inclusion} and solve the inclusion problem 
\begin{equation}\label{dr-inclusion-precond}
    \text{find }x\in\H \text{ such that } 0\in \Delta(A+B)x,
\end{equation}
where $\Delta:\H\to\H$ is a linear, invertible, bounded, positive semi-definite and self-adjoint preconditioner and naturally, we can also introduce varying preconditioners $(\Delta_k)_{k\in\N}$ here as well, leading us to the varying preconditioned DR iteration
\begin{equation}\label{dr-iteration-precond}
    w^0\in\H,\quad w^{k+1}= w^k + J_{\Delta_k B}(2J_{\Delta_k A}w^k - w^k) - J_{\Delta_k A}w^k.
\end{equation}
Notice that $\Delta_k = \I$ results in the DR method \eqref{dr-iteration-unconditioned}. Furthermore, $\Delta_k \equiv t\,\I$ or $\Delta_k = t_k \I$ allows to introduce a stepsize $t > 0$ or a stepsize sequence $(t_k)_{k\in\N}\subset \R_{>0}^\N$.

To get even more concrete, we consider minimization problems of the form
\begin{equation}\label{primal-problem}
\underset{x\in\R^n}{\text{min}}\, \Big\{f(x) + g(Kx)\Big\},
\end{equation}
for two proper, convex and lower semicontinuous functions $f:\mathbb R^n\to \overline{\mathbb R} := \mathbb R\cup \{\infty\},\, g:\mathbb R^m \to \overline{\mathbb R}$ and $K\in\mathbb R^{m\times n}$. The primal-dual optimality conditions for this problems are (under mild regularity assumptions~\cite{bauschke})
\begin{equation}\label{optimal-first}
    0 \in \partial f( x) + K^T \partial g (K  x)\quad\text{and}\quad 0\in -K \partial f^*(-K^T  y) + \partial g^* (y),
\end{equation}
where the dual variable $y\in\R^m$ is the solution of the dual problem to \eqref{primal-problem}. Furthermore $g^*$ denotes the Fenchel conjugate of $g$, which is
$$g^*(y) := \sup_{z \in \R^m} \langle z,y \rangle - g(z).$$
Both optimality conditions \eqref{optimal-first} can be combined into the single condition
\begin{equation}\label{optimal-splitting}
    0\in \biggl(
    \underbrace{
    \begin{bmatrix}
        \partial f & 0\\ 
        0 & \partial g^*
    \end{bmatrix}
    }_{=: A}
    +
    \underbrace{
    \begin{bmatrix}
        0 & K^T\\
        -K & 0
    \end{bmatrix}
    }_{=:B}
    \biggr)
    \begin{bmatrix}
    x\\y
    \end{bmatrix}.
\end{equation}
The DR method can be applied to this splitting and has been investigated~\cite{oconnor, procond-dr}.

\subsection{Related works}

Proximal point algorithms with non-stationary stepsize are known from a long time \cite{Guler1991,Rockafellar1976}. The idea to change the metric at every iteration comes from other first-order methods, such as gradient descent, where Newton metrics or quasi-Newton metrics can be employed to drastically accelerate convergence~\cite{davidon1959variable}. Examples of non-stationary preconditioned proximal point algorithms can be found in \cite{Parente2008,burke1999variable-metric-pp}, for example, and versions with additional forward term exist as well~\cite{Bonettini2016}.

The stepsize has an important role in splitting methods and it is empirically observed that there often is a ``sweet spot'' for good stepsizes~\cite{adaptive-stepsizes}. While non-stationary methods  have been under investigation in~\cite{combettes2014variable,salzo2017variable,davis2015convergence,Liang2017,adaptive-stepsizes,frankel2015splitting,tran2020non} there is less work on stepsize heuristics and adaptive stepsize selection. 
Some general rules for constant stepsizes are given in~\cite{giselsson2017tightglobalrates,moursi2019douglas} (and these rules are based on further properties of the operators such as strong monotonicity, Lipschitz continuity, and coercivity)
A heuristic stepsize rule for constant stepsizes (motivated by quadratic problems) is derived in~\cite{giselsson2017linear} and a 
self-adaptive stepsize for ADMM (which is equivalent to DR by duality) is proposed in \cite{he2000alternating}.
In~\cite{xu2017adaptive2,xu2017adaptive} the authors proposed adaptive update rules for stepsizes in ADMM based on a spectral estimation.
In \cite{lin2011linearized}, the authors proposed a nonincreasing adaptive rule for the penalty parameter in ADMM. 
Another update rule for ADMM can be found in \cite{song2016fast}.
Adaptive rules for the DR method are scarce and the only work we are aware of (in the context of monotone inclusions) is~\cite{adaptive-stepsizes}.

In this work we apply an adaptive Dogulas-Rachford method to recover points satisfying specific primal-dual optimality conditions that arise from minimization problems involving compositions of convex functions with linear terms and from saddle point problems. In this context, there have been some analysis on how to choose adaptively the stepsizes for the celebrated primal-dual hybrid gradient method, also known as Chambolle-Pock method \cite{CP2011}, namely the works
\cite{Goldstein2013,Yokota2017}. For a variant including forward steps a stepsize heuristic has been proposed in \cite{Vladarean2021}. The recent work \cite{Chambolle2023} expands the analysis of adaptive stepsizes to a stochastic version of the algorithm.

\subsection{Paper organization and contribution}
This paper starts with investigation of the degenerate variable metric proximal point algorithm in Section \ref{section-convergence}. The section's outcome is the proof of weak convergence of \eqref{vary-precond-iter}, which will be accomplished with Theorem \ref{main-convergence-theorem} followed by Corollary \ref{convergence-corollary}. The given proof is inspired by the proceeding in \cite{degenerate-pppa}, combined with ideas from \cite{quasi-fejer, warped-resolvents}.

During the first half of Section \ref{section-dr}, the connection between the preconditioned DR method \eqref{dr-iteration-precond} with the convergence results from Section \ref{section-convergence} will be provided. Therefore, a convergence proof for the varying preconditioned DR method \eqref{dr-iteration-precond} will be given. The second half of Section \ref{section-dr} deals with the application of the DR method to the minimization problem \eqref{primal-problem} using the primal-dual operator splitting \eqref{optimal-splitting}. Furthermore, the idea from \cite{oconnor} to benefit from two instead of one independent stepsizes will be extended from the stationary iterations in \cite{oconnor} to non-stationary iterations.

The newly gained freedom to choose two stepsizes in a varying way naturally leads to the question of how to choose them. An attempt to answer this question will be taken in Section \ref{section-dr-stepsize}, where the idea of adaptive stepsizes for the DR-method from \cite{adaptive-stepsizes} will be extended to two varying stepsize sequences. 

The paper will be completed with numerical examples in which the previously attained rule on how to choose stepsize sequences will be applied to exemplary problems.

\section{The degenerate variable metric proximal point algorithm and its convergence}
\label{section-convergence}
This section will provide the weak convergence of the degenerate variable metric proximal point algorithm \eqref{vary-precond-iter} to a solution of the monotone inclusion \eqref{pp-inclusion}.

In the following $\M,\M_k:\H\to\H$ will always denote linear, bounded and positive semi-definite operators. Therefore, the bilinear form $\langle u, v\rangle_\M := \langle\M u, v\rangle$ is a semi-inner product and the induced semi-norm may be denoted by $\Vert u\Vert_\M :=\langle u, u\rangle_\M^\frac{1}{2}$.

\begin{definition}[$\M$-monotonicity]
    Let $\M: \H\to\H$ be a linear, bounded and positive semi-definite operator. Then $T:\H\tto\H$ is called $\M$-monotone if
    $$\langle u_1 - u_2, v_1 - v_2\rangle_\M\geq 0,\quad\forall (u_1, v_1), (u_2, v_2)\in\G[T].$$
\end{definition}

The $\M$-monotonicity of $\M^{-1} T$ is the key to ensure the equivalent of firmly non-expansiveness in the $\M$-semi-norm context as proven by \cite[Lemma 2.5]{procond-dr}:
\begin{lemma}\label{lemma-expansiveness}
    Let $\M$ be an admissible preconditioner for an operator $T:\H\tto\H$, such that $\M^{-1} T$ is $\M$-monotone. Then $J_T^\M$ is $\M$-firmly non-expansive, i.e. it holds for all $u_1,u_2\in\H$ that
    $$\Vert J_T^\M u_1 - J_T^\M u_2\Vert^2_\M + \Vert (\I - J_T^\M)u_1 - (\I - J_T^\M)u_2\Vert^2_\M \leq \Vert u_1 - u_2 \Vert^2_\M.$$
\end{lemma}

The preparation of the convergence proof starts with two auxiliary results. The first of them being a version of \cite[Proposition 2.3]{degenerate-pppa} and for which a proof is sketched in the reference.
\begin{proposition}\label{decomposition-proposition}
    Let $\M:\H\to\H$ be a linear, bounded, self-adjoint and positive semi-definite operator. Then there exists a bounded and injective Operator $\C: \mathcal D \to \H$, where $\mathcal D$ is some real Hilbert space, such that $\M= \C\C^*$. Moreover, if $\M$ has closed range, then $\C^*$ is onto.
\end{proposition}

We will also use the following result (which follows from~\cite[Lemma 5.31]{bauschke} by setting $\epsilon_n \equiv 0$):
\begin{lemma}\label{pre-convergence-lemma}
    Let $(\alpha_k)_{k\in\N}$, $(\beta_k)_{k\in\N}$ and $(m_k)_{k\in\N}$ be sequences in $\R_{\geq 0}$, such that $\sum_{k\in\N} m_k < \infty$. If it holds for all $k\in\N$ that
    \begin{equation}\label{pre-convergence-inequality}
        \alpha_{k+1} \leq (1 + m_k)\alpha_k - \beta_k,
    \end{equation}
    then $(\alpha_k)_{k\in\N}$ converges and $\sum_{k\in\N} \beta_k < \infty$.
\end{lemma}

Now we state the main lemma:
\begin{lemma}\label{preparing-lemma}
    Let $ T:\H\tto\H$ such that $\zer T \neq \emptyset$. Assume all $\M_k$ to be admissible preconditioners for $ T$ which satisfy
    $$\M_k \to \M,\quad\quad \sum_{k\in\N}\Vert \M_{k+1} - \M_k\Vert < \infty.$$
    Assume for all $k\in\N$ that $\M_k^{-1} T$ are $\M_k$-monotone and $(\M_k +  T)^{-1}$ are $L$-Lipschitz. Let $(u^k)_{k \in\N}$ be generated by
    $$u^0 \in \H,\quad u^{k+1} = J_{ T}^{\M_k} u^k.$$
    Then $(u^k)_{k\in\N}$ is bounded and $(\Vert u^k - u^*\Vert_{\M_k})_{k\in\N}$ converges for all $u^*\in\fix J_{ T}^\M$. Furthermore,
    \begin{align*}
      \sum\limits_{k=0}^{\infty}\Vert J_{ T}^{\M_k} u^k - u^k\Vert_{\M_k} < \infty,
    \end{align*}
    i.e. $\lim_{k\to\infty}\Vert J_{ T}^{\M_k} u^k - u^k\Vert_{\M_k} = 0$.
\end{lemma}
\begin{proof}
    Let $\M_k = \C_k\C^*_k$ be a decomposition of $\M_k$ according to Proposition \ref{decomposition-proposition}. Since all $(\M_k +  T)^{-1}$ are $L$-Lipschitz, it holds for all $u, \tilde u \in \H$ and $k\in\N$ that
    \begin{equation}\label{conv-thm-lipschitz}
        \begin{split}
            \Vert J^{\M_k}_{ T} u - J^{\M_k}_{ T} \tilde u \Vert
            &= \Vert (\M_k +  T)^{-1}\C_k\C^*_k u - (\M_k +  T)^{-1}\C_k\C_k^* \tilde u\Vert\\
            &\leq L \Vert \C_k\Vert \Vert \C^*_k (u - \tilde u)\Vert\\
            &= L \sqrt{\Vert \C_k\Vert^2} \langle \C^*_k(u - \tilde u), \C^*_k(u-\tilde u)\rangle^{\frac{1}{2}} \\
            &= L \sqrt{\Vert \C_k\C_k^*\Vert} \langle \M_k(u - \tilde u),u-\tilde u\rangle^{\frac{1}{2}} \\
            &= L \sqrt{\Vert \M_k\Vert} \Vert u - \tilde u\Vert_{\M_k}.
        \end{split}
    \end{equation}
    Since $(\M_k)_{k\in\N}$ is convergent, $\Vert\M_k\Vert$ are bounded and there exists $C > 0$, such that
    \begin{equation}\label{conv-thm-lip-const}
        \Vert J^{\M_k}_{ T} u - J^{\M_k}_{ T} \tilde u \Vert \leq C \Vert u - \tilde u\Vert_{\M_k}.
    \end{equation}
    The combination of this inequality with the $\M_k$-firmly-non-expansiveness of $J_T^{\M_k}$ provided by Lemma \ref{lemma-expansiveness} yields for all $u^* \in\fix J_T^\M = \fix J_T^{\M_k}$ that
    \begin{alignat*}{3}\label{conv-thm-long-eqn}
        \Vert &&u^{k+1} - u^*&\Vert^2_{\M_{k+1}} = \Vert J_T^{\M_k}u^k - J_T^{\M_k}u^*\Vert^2_{\M_{k+1}}\\
        &&=\quad &\Vert J_T^{\M_k}u^k - J_T^{\M_k}u^*\Vert^2_{\M_k} + \Vert J_T^{\M_k}u^k - J_T^{\M_k}u^*\Vert^2_{\M_{k+1} - \M_k}\\
        && \overset{\ref{lemma-expansiveness}}{\leq} \;\;&\Vert u^k - u^* \Vert^2_{\M_k} - \Vert J_{ T}^{\M_k} u^k - u^k\Vert^2_{\M_k} + \Vert J_T^{\M_k}u^k - J_T^{\M_k}u^*\Vert^2_{\M_{k+1} - \M_k}\\
        &&\leq\quad &\Vert u^k - u^* \Vert^2_{\M_k} - \Vert J_{ T}^{\M_k} u^k - u^k\Vert^2_{\M_k} + \Vert \M_{k+1}-\M_k\Vert \Vert J_T^{\M_k}u^k - J_T^{\M_k}u^*\Vert^2\\
        && \overset{\eqref{conv-thm-lip-const}}{\leq}\quad &\Vert u^k - u^* \Vert^2_{\M_k} - \Vert J_{ T}^{\M_k} u^k - u^k\Vert^2_{\M_k} + C^2 \Vert \M_{k+1}-\M_k\Vert \Vert u^k - u^* \Vert^2_{\M_k}\\
        &&= \quad&\left(1 + C^2 \Vert \M_{k+1}-\M_k\Vert\right) \Vert u^k - u^* \Vert^2_{\M_k}- \Vert J_{ T}^{\M_k} u^k - u^k\Vert^2_{\M_k}.
    \end{alignat*}
    An application of Lemma \ref{pre-convergence-lemma} with $\alpha_k =  \Vert u^k - u^* \Vert^2_{\M_k}$, $\beta_k = \Vert J_{ T}^{\M_k} u^k - u^k\Vert^2_{\M_k}$ and $m_k = C^2 \Vert \M_{k+1}-\M_k\Vert$ yields the convergence of $(\Vert u^k - u^* \Vert_{\M_k})_{k\in\N}$ as well as the summability condition $\sum\limits_{k=1}^{\infty}\Vert J_{ T}^{\M_k} u^k - u^k\Vert^2_{\M_k} < \infty$.
    
    The convergence also implies the boundedness of $(\Vert u^k - u^* \Vert^2_{\M_k})_{k\in\N}$ and since \eqref{conv-thm-lip-const} also yields
    $$\Vert u^{k+1} - u^* \Vert \leq C \Vert u^k - u^*\Vert_{\M_k},$$
    the sequence $(u^k)_{k\in\N}$ is bounded, too.
\end{proof}

\begin{theorem}\label{main-convergence-theorem}
    Let $ T:\H\tto\H$ such that $\zer T \neq \emptyset$. Assume all $\M_k$ and $\M$ to be admissible preconditioners for $ T$ which satisfy
    $$\M_k \to \M,\quad\quad \sum_{k\in\N}\Vert \M_{k+1} - \M_k\Vert < \infty.$$
    Furthermore, assume for all $k\in\N$ that $\M_k^{-1} T$ are $\M_k$-monotone and $(\M_k +  T)^{-1}$ are $L$-Lipschitz. Let $(u^k)_{k \in\N}$ be generated by
    $$u^0 \in \H,\quad u^{k+1} = J_{ T}^{\M_k} u^k.$$
    If every weak cluster point of $(u^k)_{k\in\N}$ is a fixed point of $J_T^\M$, the sequence $(u^k)_{k\in\N}$ converges weakly to some $u \in \zer  T$.
\end{theorem}
\begin{proof}
    Let $u^*$ be a fixed point of $J_T^\M$. Then, according to Lemma \ref{preparing-lemma}, the sequence $(\Vert u^{k} - u^* \Vert^2_{\M_k})_{k\in\N}$ is convergent.\\
    Since Lemma \ref{preparing-lemma} also provides the boundedness of $(u^k)_{k\in\N}$, this sequence has at least one weak cluster point $u^* \in\fix J_T^\M$ by assumption. Let $u^{**}$ to be another cluster point of $(u^k)_{k\in\N}$ with $u^*\neq u^{**}$. Both $u^*$ and $u^{**}$ are fixed points of all $J_T^{\M_k}$ and $J_T^{\M}$. Furthermore, there exist subsequences $(u^{k_i})_{i\in\N}$ and $(u^{k_j})_{j\in\N}$, which converge weakly to $u^*$ and $u^{**}$, i.e. $u^{k_i}\weakto u^*$ and $u^{k_j}\weakto u^{**}$.\\
    Consider the inner product
    \begin{equation}\label{conv-thm-inner-prod}
        \langle \M_k u^k, u^* - u^{**}\rangle = \frac{1}{2} \big(\Vert u^* - u^{**}\Vert^2_{\M_k} - \Vert u^k - u^* \Vert_{\M_k}^2+ \Vert u^*\Vert^2_{\M_k} - \Vert u^{**}\Vert^2_{\M_k} \big).
    \end{equation}
    The convergence of the second norm on the right hand side of \eqref{conv-thm-inner-prod} is already proven in Lemma \ref{preparing-lemma}. The remaining norms converge by the assumption of $\M_k \to \M$. Hence, both the right hand side and the inner product on the left hand side converge. The attempt of calculating this limit for both subsequences $(u^{k_i})_{i\in\N}$ and $(u^{k_j})_{j\in\N}$ results under usage of $u^{k_i} \weakto u^*$, $u^{k_j} \weakto u^{**}$ and $\M_k (u^* - u^{**})\to \M(u^*-u^{**})$ in two limits
    \begin{equation*}
        \begin{split}
            \ell &= \lim_{i\to\infty} \langle \M_{k_i} u^{k_i}, u^* - u^{**}\rangle= \lim_{i\to\infty} \langle u^{k_i}, \M_{k_i}( u^* - u^{**})\rangle = \langle u^*, u^* - u^{**}\rangle_\M,\\
            \tilde \ell &= \lim_{j\to\infty} \langle \M_{k_j} u^{k_j}, u^* - u^{**}\rangle = \lim_{j\to\infty} \langle u^{k_j}, \M_{k_j}( u^* - u^{**})\rangle = \langle u^{**}, u^* - u^{**}\rangle_\M.
        \end{split}
    \end{equation*}
    The uniqueness of the limit enforces $\ell = \tilde \ell$ and therefore
    $$0 = \ell - \tilde \ell = \langle u^* - u^{**}, u^* - u^{**}\rangle_\M = \Vert u^* - u^{**}\Vert^2_\M.$$
    This implies $\M u^* = \M u^{**}$ and in particular
    \begin{equation*}
            u^* = J_T^\M u^* = (\M +  T)^{-1}\M u^* = (\M +  T)^{-1}\M u^{**}= J_T^\M u^{**} = u^{**}.
    \end{equation*}
    Thus, all weak cluster points of $(u^k)_{k\in\N}$ coincide and $(u^k)_{k\in\N}$ converges weakly to a fixed point of $J_{ T}^\M$.\\
\end{proof}

\begin{corollary}\label{convergence-corollary}
    Let $ T:\H\tto\H$ be a maximal monotone operator such that $\zer  T  \neq \emptyset$. Assume all $\M_k$ and $\M$ to be admissible preconditioners for $ T$, which satisfy
    $$\M_k \to \M,\quad\quad \sum_{k\in\N}\Vert \M_{k+1} - \M_k\Vert < \infty.$$
    Furthermore, assume for all $k\in\N$ that $(\M_k +  T)^{-1}$ are $L$-Lipschitz. Let $(u^k)_{k \in\N}$ be generated by
    $$u^0 \in \H,\quad u^{k+1} = J_{ T}^{\M_k} u^k.$$
    Then $(u^k)_{k\in\N}$ converges weakly to some $u \in \zer  T$.
\end{corollary}
\begin{proof}
    Let $(u, v),(\tilde u,\tilde v)\in \G[\M_k^{-1} T]$ and consequently $(u,\M_k v),(\tilde u, \M_k \tilde v)\in\G[T]$. The monotonicity of $ T$ shows that
    $$\langle v - \tilde v, u - \tilde u\rangle_{\M_k} = \langle \M_k v - \M_k \tilde v , u - \tilde u \rangle \geq 0.$$
    Hence, $\M_k^{-1} T$ are $\M_k$-monotone. Now, the claim follows with Theorem \ref{main-convergence-theorem} if every weak cluster point of $(u^k)_{k\in\N}$ is a fixed point of $J_T^\M$. Hence, assume $u^{k_{i}+1} = J_T^{\M_{k_i}}u^{k_i}\weakto u\in\H$. Lemma \ref{preparing-lemma} provides the convergence
    $$\lim_{i\to\infty} \Vert J_T^{\M_{k_i}} u^{k_i} - u^{k_i}\Vert_{\M_{k_i}} = 0,$$
    which (together with boundness of $\|\C_{k_i}\|$) implies that
    $$ T J_T^{\M_{k_i}}u^{k_i}\ni\M_{k_i}(u^{k_i} - J^{\M_{k_i}}_{ T} u^{k_i})\to 0.$$
    The maximality of $ T$ enforces $ T$ to be closed in $\H_{\text{weak}} \times \H_{\text{strong}}$ (see \cite[Proposition 20.38]{bauschke}). Hence, we have $0\in T u$ or equally $u \in \zer  T$.\\
\end{proof}

\section{Preconditioned Douglas-Rachford method}\label{section-dr}
\subsection{Douglas-Rachford method as degenerate variable metric proximal point algorithm}
For the following derivation of the DR method from the perspective of the degenerate variable metric proximal point algorithm, a more general version of the well known Moreau decomposition  will be used.
\begin{proposition}[Generalized Moreau decomposition]\label{general-moreau-decomposition}
    Let $T:\H\tto\H$ be a set valued operator and $\Sigma:\H\to\H$ linear and invertible, such that $J_{\Sigma^{-1}T}$ and $J_{\Sigma T^{-1}}$ are everywhere defined and single valued. Then it holds that
    $$J_{\Sigma T^{-1}}(x) = x - \Sigma J_{\Sigma^{-1}T}(\Sigma^{-1}x).$$
\end{proposition}
\begin{proof}
    
     Let $y= J_{\Sigma T^{-1}}(x) = (\I + \Sigma T^{-1})^{-1}x$, then
    \begin{alignat*}{3}
        && (\I + \Sigma T^{-1})^{-1}x &= y\\
        \Leftrightarrow && x &\in y + \Sigma T^{-1}y\\
        \Leftrightarrow && y &\in T( \Sigma^{-1}(x-y))\\
        \Leftrightarrow && \Sigma^{-1}x &\in \Sigma^{-1}T( \Sigma^{-1}(x-y)) + \Sigma^{-1}(x - y)\\
        \Leftrightarrow && \Sigma^{-1}x &\in (\I + \Sigma^{-1}T)(\Sigma^{-1}(x-y))\\
        \Leftrightarrow && x-y &\in \Sigma J_{\Sigma^{-1}T}(\Sigma^{-1}x).
    \end{alignat*}
   Since $J_{\Sigma^{-1} T}$ is single valued, then the last line is an equality and the claim follows.
\end{proof}

Some $x\in\H$ satisfies the DR inclusion problem \eqref{dr-inclusion} if there exists $y\in\H$, such that
$$y\in Bx,\quad 0\in Ax + y.$$
The first inclusion is equivalent to $0\in -x + B^{-1}y$ (where $B^{-1}$ does always exist in the set valued sense). Thus, both inclusions can be written using a block operator as
\begin{equation}
    \begin{bmatrix}0\\0\end{bmatrix}\in
    \begin{bmatrix}
        A & \I\\
        -\I & B^{-1}
    \end{bmatrix}
    \begin{bmatrix}x\\y\end{bmatrix}
    =: T u,
\end{equation}
where $T$ is defined on $\H\times\H$.
Let $(\Delta_k)_{k\in\N}$ be a sequence of linear, invertible, bounded, positive semi-definite and self-adjoint operators on $\H$. The proximal point iteration \eqref{vary-precond-iter} for $T:\H\times\H\tto\H\times\H$, combined with the varying preconditioner
\begin{equation}\label{varyig-preconditioner}
    \M_k := \begin{bmatrix}
        \Delta_k^{-1} & -\I\\
        -\I & \Delta_k
    \end{bmatrix},
\end{equation}
is
\begin{alignat}{3}\label{reformulation-one}
    &&u^{k+1} &= (\M_k + T)^{-1}\M_k u^k\\
    \Leftrightarrow&&\quad \M_k u^k &\in (\M_k + T)\, u^{k+1}\\
    \Leftrightarrow&&\quad \begin{bmatrix} 
    \Delta_k^{-1} & - \I\\ - \I & \Delta_k
    \end{bmatrix}
    \begin{bmatrix}
    x^k\\y^k
    \end{bmatrix}
    &\in
    \begin{bmatrix}
    \Delta_k^{-1} + A & 0 \\ -2\I & \Delta_k  + B^{-1}
    \end{bmatrix}
    \begin{bmatrix}
    x^{k+1}\\ y^{k+1}
    \end{bmatrix}\\
    \Leftrightarrow&&\quad
    \begin{bmatrix}
    x^k - \Delta_k y^k\\
     - \Delta_k^{-1}(x^k - \Delta_ky^k)
    \end{bmatrix}
    &\in
    \begin{bmatrix}
    (\I + \Delta_k A)x^{k+1}\\
    -2\Delta_k^{-1}x^{k+1} + (\I + \Delta_k^{-1} B^{-1})y^{k+1}
    \end{bmatrix}
\end{alignat}

Under the assumption that $J_{\Delta_kA}$, $J_{\Delta_k^{-1}B^{-1}}$ and $J_{\Delta_kB}$ are defined everywhere and single valued, it follows that
\begin{equation}
    \begin{split}
        &\begin{bmatrix}
            x^{k+1}\\y^{k+1}
        \end{bmatrix}
        = \begin{bmatrix}
            J_{\Delta_kA}(x^k - \Delta_k y^k)\\
            J_{\Delta_k^{-1} B^{-1}}(2\Delta_k^{-1} x^{k+1} + \Delta_k^{-1}(x^k - \Delta_k y^k))
        \end{bmatrix}\\
        &\quad\overset{\text{Prop.\ref{general-moreau-decomposition}}}{=}
        \begin{bmatrix}
            J_{\Delta_kA}(x^k - \Delta_k y^k)\\
            2\Delta_k^{-1} x^{k+1} + \Delta_k^{-1}(x^k - \Delta_k y^k) - \Delta_k^{-1}J_{\Delta_k B}(2 x^{k+1} + (x^k - \Delta_k y^k))
        \end{bmatrix}.
    \end{split}
\end{equation}
The substitution of $w^k := x^k - \Delta_k y^k$, for all $k$, leads to the iteration
\begin{equation}\label{dr-with-delta}
    w^{k+1}= w^k + J_{ \Delta_k B}(2J_{ \Delta_k A}w^k - w^k) - J_{\Delta_k  A}w^k.
\end{equation}
This is the varying preconditioned DR iteration \eqref{dr-iteration-precond}. Choosing the preconditioners as $\Delta_k \equiv \I$ results in the standard DR-iteration as from \cite{lions-mercier}. A stepsize $t>0$ can be added by setting $\Delta_k = t \I$ or $\Delta_k = t_k \I$ in order to get a non-stationary iteration with positive stepsizes $(t_k)_{k\in\N}$.
In all of these cases iteration \eqref{dr-with-delta}, if it converges, does not converge to a solution of $0\in(A+B)x$, but the sequence $(J_{\Delta_k A} w^k)_{k\in\N}$ does.

\begin{proposition}\label{convergence-precond-dr}
    Let $A, B:\H\tto\H$ be maximal monotone operators such that $\zer (A+B)\neq \emptyset$. Define $(\Delta_k)_{k\in\N}$ with $\Delta_k:\H\to\H$ and $\Delta:\H\to\H$ such that they are linear, uniformly bounded, uniformly boundedly invertible (i.e. $\Vert\Delta_{k}\Vert\leq \overline{C}$ and $\Vert \Delta_{k}^{-1}\Vert\leq \underline{C}$ for some $\overline{C},\underline{C}$ and all $k$), 
    positive definite and selfadjoint. Additionally let them satisfy
    $$\Delta_k \to \Delta,\qquad\sum_{k\in\N}\Vert \Delta_{k+1}-\Delta_k\Vert < \infty.$$
    Let $(w^k)_{k\in\N}$ be generated by
    $$w^0,\in \H \quad w^{k+1}= w^k + J_{ \Delta_k B}(2J_{ \Delta_k A}w^k - w^k) - J_{\Delta_k  A}w^k.$$
    Then the sequence $(J_{\Delta_k A}w^k)_{k\in\N}$ converges weakly to some $x\in\zer (A+B)$.
\end{proposition}
\begin{proof}
    It has already been shown that the varying preconditioned DR method can be realized by the varying PPP algorithm with the choice of $T: \H\times\H \tto \H\times \H$ and $\M_k:\H\times\H \to \H\times \H$ as
    \begin{equation*}
        T = \begin{bmatrix} A & \I \\ -\I & B^{-1}\end{bmatrix},\quad\quad \M_k = \begin{bmatrix}\Delta_k^{-1} & -\I\\ -\I & \Delta_k
        \end{bmatrix}.
    \end{equation*}
    Where the assumption on the resolvents is satisfied since by assumption $A$ and $B$ are maximal monotone. In fact, we have that $\Delta_k A, \Delta_k B$ are maximal monotone in $(\H,\|.\|_{\Delta_k^{-1}})$ and $\Delta_k^{-1} B^{-1}$ is maximal monotone in $(\H,\|.\|_{\Delta_k})$ so that the resolvents $J_{\Delta_kA}$, $J_{\Delta_k^{-1}B^{-1}}$ and $J_{\Delta_kB}$ are defined everywhere and single valued.
    The maximal monotonicity of $T$ is a direct consequence of the maximal monotonicity of $A$ and $B$ and from $\zer (A+B)\neq \emptyset$ follows $\zer T \neq \emptyset$. Hence, the proposition follows with Corollary \ref{convergence-corollary} after verifying that all of the following conditions are fulfilled.\medskip
    
    \noindent\textbf{1. $\M_k$ are admissible preconditioners for $T$:}
    The linearity of $\M_k$ is clear by definition. For $x = [x_1, x_2]^T$ it holds that
    \begin{equation*}
        \begin{split}
            \Vert \M_k x\Vert_{\H\times\H} &\leq \Vert \begin{bmatrix}\Delta_k^{-1}\,\I & 0\\  0 & \Delta_k\,\I\end{bmatrix}\begin{bmatrix}x_1\\x_2\end{bmatrix}\Vert_{\H\times\H} + \Vert \begin{bmatrix}0 & -\I\\  -\I & 0\end{bmatrix}\begin{bmatrix}x_1\\x_2\end{bmatrix}\Vert_{\H\times\H}\\
            &= \sqrt{\Vert \Delta_k^{-1}x_1\Vert^2 + \Vert \Delta_k x_2\Vert^2} + \Vert x\Vert_{\H\times\H} \leq (\sqrt{\underline{C}^{2} + \overline{C}^2}+1)\Vert x\Vert_{\H\times\H}.
        \end{split}
    \end{equation*}
    Hence, the $\M_k$ are uniformly bounded. Since the $\Delta_k$ are self-adjoint we get
    \begin{equation*}
        \begin{split}
            \langle \M_k x, x \rangle_{\H\times\H} &= \langle \begin{bmatrix}\Delta_k^{-1} x_1 - x_2\\ -x_1 + \Delta_kx_2\end{bmatrix}, \begin{bmatrix}x_1\\x_2\end{bmatrix}\rangle_{\H\times\H}\\
            &= \langle \Delta_k^{-1}x_1, x_1\rangle - 2 \langle x_1, x_2 \rangle + \langle \Delta_k x_2, x_2 \rangle\\
            &= \Vert \Delta_k^{-1}x_1 \Vert^2_{\Delta_k} - 2 \langle \Delta_k^{-1}x_1, x_2\rangle_{\Delta_k} + \Vert x_2 \Vert^2_{\Delta_k}\\
            &= \Vert \Delta_k^{-1} x_1 - x_2\Vert^2_{\Delta_k}  \geq 0.
        \end{split}
    \end{equation*}
    Therefore, all $\M_k$ are positive semi-definite.
    Now we verify that the $(\M_k+T)^{-1}$ are single valued and have full domain. These inverses are
    \begin{equation}\label{lipschitz-preparation}
        \begin{split}
            (\M_k + T)^{-1} &= \begin{bmatrix}\Delta_k^{-1} + A & 0\\ -2\I & \Delta_k + B^{-1}\end{bmatrix}^{-1}\\
            &= \begin{bmatrix}\Delta_k^{-1}(\I + \Delta_k A) & 0\\ -2\I & \Delta_k (\I + \Delta_k^{-1}B^{-1})\end{bmatrix}^{-1}.
        \end{split}
    \end{equation}
    Such lower triangular block operators are invertible and single valued, if the operators on their diagonals are invertible. As mentioned at the beginning of the proof, from maximal monotonicity of $A$ and $B$ we have that the resolvents $(\I + \Delta_k A)^{-1}$ and $(\I + \Delta_k^{-1} B^{-1})^{-1}$ are single valued and everywhere defined.
    Hence, the inverses $(\M_k+T)^{-1}$ are single valued and have full domain, which implies that $\M_k$ are admissible preconditioners for $T$.\medskip
    
    \noindent\textbf{2. $\M_k$ and $\M$ fulfil convergence requirements:}
    The convergence of $\M_k\to\M$ with $\M$ defined as $\M:= \begin{bmatrix}\Delta^{-1} & -\I\\-\I & \Delta\end{bmatrix}$ is fulfilled since $\Delta_k\to\Delta$. Furthermore, there exists a constant $C>0$, such that
    \begin{equation*}
        \begin{split}
            \Vert \M_{k+1}-\M_k\Vert &= \Vert\begin{bmatrix}(\Delta_{k+1}^{-1}-\Delta_k^{-1}) & 0\\ 0 & (\Delta_{k+1}-\Delta_k)\end{bmatrix}\Vert\\
            &\leq \Vert\begin{bmatrix}-\Delta_{k+1}^{-1}\Delta_k^{-1}&0\\0&\I\end{bmatrix}\Vert\cdot \Vert \Delta_{k+1}-\Delta_k\Vert\\
            & \leq C\cdot\Vert \Delta_{k+1}-\Delta_k\Vert.
        \end{split}
    \end{equation*}
    From the sumability of $\Vert \Delta_{k+1}-\Delta_k\Vert$ now follows that
    $$\sum_{k\in\N}\Vert \M_{k+1}-\M_k\Vert \leq C \cdot \sum_{k\in\N}\Vert \Delta_{k+1}-\Delta_k\Vert < \infty.$$
    \medskip
    
    \noindent\textbf{3. $(\M_k + T)^{-1}$ are $L$-Lipschitz:}
    After defining $y=[y_1, y_2]^T$, one sees from \eqref{lipschitz-preparation} after executing the inverse that
    $$y = (\M_k+T)^{-1}x \quad \Leftrightarrow \quad x
            \in
            \begin{bmatrix}
                \Delta_k^{-1}(\I + \Delta_k A)y_1\\
                -2 y_1 + \Delta_k(\I + \Delta_k^{-1} B^{-1})y_2
            \end{bmatrix}
    $$
    The maximal monotonicity of $\Delta_k A$ and $\Delta_k^{-1}B^{-1}$ in $(\H,\|.\|_{\Delta_k^{-1}})$ and $(\H,\|.\|_{\Delta_k})$ respectively, implies the single valuedness of the resolvents in
    \begin{equation*}
        (\M_k + T)^{-1}x = 
        \begin{bmatrix}
            y_1\\y_2
        \end{bmatrix}
        = \begin{bmatrix}
            J_{\Delta_kA}(\Delta_k x_1)\\
            J_{\Delta_k^{-1} B^{-1}}(\Delta_k^{-1}(2 J_{\Delta_kA}(\Delta_k x_1) + x_2))
        \end{bmatrix}.
    \end{equation*}
    Since the resolvents of maximal monotone operators are non-expansive and all $\Delta_k$ and $\Delta_k^{-1}$ can be bounded by the same constant, one may also find a constant $L>0$ such that $(\M_k + T)^{-1}$ are $L$-Lipschitz.
\end{proof}

As a corollary of this general result we get the following result which has also been shown in~\cite{Liang2017,adaptive-stepsizes}:

\begin{corollary}\label{varying-dr-stepsize-convergence}
Let $A,B:\H\tto\H$ be maximal monotone and $\zer (A+B) \neq \emptyset$. Let the DR iteration be given as
\begin{equation}\label{dr-iteration-in-propos}
    w^0\in\H,\quad w^{k+1}= w^k + J_{t_k B}(2J_{t_k A}w^k - w^k) - J_{t_k A}w^k,
\end{equation}
where $(t_k)_{k\in\N} \in \R_{>0}^\N$ is a non-negative stepsize sequence with
$$t_k \to t > 0,\qquad \sum_{k\in\N} \vert t_{k+1} - t_k \vert < \infty.$$
Then, the sequence $(J_{t_k A}w^k)_{k\in\N}$ converges weakly to some $x \in\zer (A+B)$.
\end{corollary}
\begin{proof}
The statement follows directly from Proposition \ref{convergence-precond-dr} after setting $\Delta_k = t_k \I$. It is easy to verify that $\Delta_k$ are linear, invertible, bounded by the same constant, positive semi-definite and self-adjoint. Furthermore, $\sum_{k\in\N}\Vert \Delta_{k+1}-\Delta_k\Vert < \infty$ follows from the given assumptions on $(t_k)_{k\in\N}$.\\
\end{proof}

\subsection{Primal-dual Douglas-Rachford}
The DR method can be used to solve the minimization problem
\begin{equation}\label{primal}
\underset{x\in\mathcal{X}}{\text{min}}\, \Big\{f(x) + g(Kx)\Big\},
\end{equation}
for $f:\mathcal{X}\to \overline{\mathbb R} := \mathbb R\cup \{\infty\},\, g: \mathcal{Y} \to \overline{\mathbb R}$, where $\mathcal{X}$ and $\mathcal{Y}$ are two Hilbert spaces and $K\in \mathcal{L}(\mathcal{X},\mathcal{Y})$. In order to apply the DR method to this problem, one considers the monotone inclusion
\begin{equation}\label{primal-dual-splitting}
    0\in \biggl(
    \underbrace{\begin{bmatrix}
        \partial f & 0\\ 
        0 & \partial g^*
    \end{bmatrix}}_{=:A}
    +
    \underbrace{\begin{bmatrix}
        0 & K^*\\
        -K & 0
    \end{bmatrix}}_{=:B}
    \biggr)
    \begin{bmatrix}
    x\\y
    \end{bmatrix},
\end{equation}
where $y$ denotes the solution of the dual problem of \eqref{primal} \cite{oconnor, procond-dr}. Applying the stationary DR iteration with stepsize $t>0$ gives
\begin{equation}\label{dr-iteration-stationary}
    \begin{cases}
        \quad x^k &= \,\,\prox_{tf}(p^k)\\
        \quad  y^k &= \,\,\prox_{t g^*}( q^k)\\
        \;\;\begin{bmatrix}
            u^k\\  v^k
        \end{bmatrix} &=
        \begin{bmatrix}
            \I & t  K^*\\
            -t K & \I
        \end{bmatrix}^{-1}

        \begin{bmatrix}
            2x^k - p^k\\
            2  y^k -  q^k
        \end{bmatrix}\\
        \quad p^{k+1} &= \,\,p^k + u^k - x^k\\
        \quad q^{k+1} &= \,\, q^k + v^k - y^k
    \end{cases}.
\end{equation}
O'Connor and Vandenberghe introduced in \cite{oconnor} the idea to consider the minimization of $f(x)+\tilde g(\tilde K x)$ instead of \eqref{primal}, where $\tilde g(y):= g(\beta^{-1} y)$ and $\tilde K := \beta K$. This does not change the problem itself, but allows to introduce a dual stepsize $s := \beta t$. The insertion of $\tilde g$ and $\tilde K$ and re-scaling of variables in \eqref{dr-iteration-stationary} results in 
\begin{equation}\label{dr-s-stationary-iteration}
    \begin{cases}
        \quad x^k &= \,\,\prox_{tf}(p^k)\\
        \quad \tilde y^k &= \,\,\prox_{s g^*}(\tilde q^k)\\
        \;\;\begin{bmatrix}
            u^k\\ \tilde v^k
        \end{bmatrix} &=
        \begin{bmatrix}
            \I & t  K^*\\
            -s K & \I
        \end{bmatrix}^{-1}
        \begin{bmatrix}
            2x^k - p^k\\
            2 \tilde y^k - \tilde q^k
        \end{bmatrix}\\
        \quad p^{k+1} &= \,\,p^k + u^k - x^k\\
        \quad \tilde q^{k+1} &= \,\,\tilde q^k +  \tilde v^k - \tilde y^k
    \end{cases},
\end{equation}
and naturally, one can consider the corresponding non-stationary scheme
\begin{equation}\label{dr-iteration}
    \begin{cases}
        \quad x^k &= \,\,\prox_{t_kf}(p^k)\\
        \quad \tilde y^k &= \,\,\prox_{s_k g^*}(\tilde q^k)\\
        \;\;\begin{bmatrix}
            u^k\\ \tilde v^k
        \end{bmatrix} &=
        \begin{bmatrix}
            \I & t_k  K^*\\
            -s_k K & \I
        \end{bmatrix}^{-1}
        \begin{bmatrix}
            2x^k - p^k\\
            2 \tilde y^k - \tilde q^k
        \end{bmatrix}\\
        \quad p^{k+1} &= \,\,p^k + u^k - x^k\\
        \quad \tilde q^{k+1} &= \,\,\tilde q^k +  \tilde v^k - \tilde y^k
    \end{cases},
\end{equation}
with stepsize sequences $(t_k)_{k\in\N}$ and $(s_k)_{k\in\N}$.

\begin{remark}
For varying stepsizes, the re-scaling argument which allowed the transition from \eqref{dr-iteration-stationary} to \eqref{dr-s-stationary-iteration}, is no longer valid. While trying to derive \eqref{dr-iteration} from a non-stationary version of \eqref{dr-iteration-stationary}, one arrives at the iteration

\begin{equation*}
    \begin{cases}
    \quad x^k &= \,\,\prox_{t_kf}(p^{k})\\
    \quad \tilde y^k &= \,\,\prox_{s_k g^*}(\tilde q^{k})\\
    \,\,\begin{bmatrix}
        u^k\\\tilde v^k
    \end{bmatrix}
    &=
    \begin{bmatrix}
        \I& t_k K^T\\-s_k K & \I
    \end{bmatrix}^{-1}
    \begin{bmatrix}
        2x^k - p^{k}\\
        2 \tilde z^k - \tilde q^{k}
    \end{bmatrix}\\
    \quad p^{k+1} &= \,\,p^{k} + u^k-x^k\\
    \quad \tilde q^{k+1} &=\,\, \frac{\beta_{k+1}}{\beta_k}(\tilde q^{k} + \tilde v^k - \tilde y^k)
\end{cases}.
\end{equation*}
Hence, the equivalence of both non-stationary iteration schemes does not hold anymore and the convergence of \eqref{dr-iteration} does not follow as for \eqref{dr-s-stationary-iteration} from the convergence of the non re-scaled iteration.
\end{remark}

\begin{remark}
  We note that one can also consider algorithm~\eqref{dr-iteration} where $\prox_{tf}$ is replaced by $J_{tA_{1}}$ and $\prox_{sg^{*}}$ by $J_{sA_{2}^{-1}}$, respectively for two maximally monotone operators $A_{1},A_{2}$ (and similarly for the non-stationary case and the case with dual stepsizes). The resulting schemes can be used to solve the inclusion $0\in A_1(x) + K^*A_2(Kx)$.
  The convergence theory that we develop also applies to this slightly more general case.
\end{remark}

The convergence of \eqref{dr-iteration} can be proven with the theory from Section \ref{section-convergence}:

\begin{proposition}\label{final-convergence-result}
    Let $(t_k)_{k\in\N} \in \R_{>0}^\N$ and $(s_k)_{k\in\N} \in \R_{>0}^\N$ be two stepsize sequences, which satisfy
    \begin{equation*}
        \begin{split}
            t_k& \to t > 0, \quad \sum_{k\in\N}\vert t_{k+1} - t_k\vert < \infty,\\
            s_k& \to s > 0, \quad \sum_{k\in\N}\vert s_{k+1} - s_k\vert < \infty.
        \end{split}
    \end{equation*}
    Then iteration \eqref{dr-iteration} converges weakly, if a finite minimizer of the primal problem \eqref{primal} exists. In that case, the occurring sequence $(x^k)_{k\in\N}$ converges weakly to such a minimizer.
\end{proposition}
\begin{proof}
    Let $\mathcal{H}:=\mathcal{X}\times \mathcal{Y}$. Choosing $A:\mathcal{H}\tto\mathcal{H}$ and $B:\mathcal{H}\to\mathcal{H}$ and a sequence of preconditioners $(\Delta_k)_{k\in\N}$ as
    \begin{equation*}
        A := \begin{bmatrix}
        \partial f & 0\\ 
        0 & \partial g^*
    \end{bmatrix},
    \qquad
    B := \begin{bmatrix}
        0 & K^T\\
        -K & 0
    \end{bmatrix},
    \qquad
    \Delta_k := \begin{bmatrix}t_k \I & 0\\ 0 & s_k \I\end{bmatrix},
    \end{equation*}
    results in the evaluation of the resolvents $J_{\Delta_k A}$ and $J_{\Delta_k B}$ as
    \begin{equation*}
        \begin{split}
            (\I + \Delta_k A)^{-1}
            \begin{bmatrix}
                 p^k \\  \tilde q^k
            \end{bmatrix}
            &= (\I + \begin{bmatrix}
                t_k \partial f & 0 \\
                0 & s_k \partial g^*
            \end{bmatrix})^{-1}
            \begin{bmatrix}
                p^k \\  \tilde q^k
            \end{bmatrix}
            =
            \begin{bmatrix}
                \prox_{t_kf}(\tilde x) \\ \prox_{s_kg^*}(\tilde y)
            \end{bmatrix},\\
            (\I + \Delta_k B)^{-1}\begin{bmatrix}
                \tilde u^k\\\tilde v^k
            \end{bmatrix}
            &= 
            \begin{bmatrix}
                \I & t_kK^T\\
                -s_k K & \I
            \end{bmatrix}^{-1}
            \begin{bmatrix}
                \tilde u^k\\\tilde v^k
            \end{bmatrix}.
        \end{split}
    \end{equation*}
    Hence, iteration \eqref{dr-iteration} is equivalent to the varying preconditioned DR iteration
    $$w^{k+1}= w^k + J_{ \Delta_k B}(2J_{ \Delta_k A}w^k - w^k) - J_{\Delta_k  A}w^k$$
    and, since $A$ and $B$ are maximal monotone operators, the convergence of \eqref{dr-iteration} can be shown with Proposition \ref{convergence-precond-dr}. In fact, it remains to show that all requirements on $(\Delta_k)_{k\in\N}$ are fulfilled. The convergence of $(\Delta_k)_{k\in\N}$ to some $\Delta$ follows from the convergence $t_k\to t$ and $s_k \to s$. Furthermore, it holds that
    \begin{equation*}
        \Vert \Delta_{k+1} - \Delta_k \Vert \leq \Vert (t_{k+1} - t_k) \I\Vert + \Vert(s_{k+1} - s_k)  \I \Vert = \vert t_{k+1}-t_k\vert + \vert s_{k+1}-s_k\vert
    \end{equation*}
    and hence
    \begin{equation*}
        \sum_{k\in\N} \Vert \Delta_{k+1} - \Delta_k \Vert \leq \sum_{k\in\N}\vert t_{k+1}-t_k\vert + \sum_{k\in\N} \vert s_{k+1}-s_k\vert < \infty.
    \end{equation*}
    Additionally, $\Delta_k$ are linear, bounded and invertible. Their symmetry implies that they are self-adjoint and their positive semi-definiteness follows from $\langle \Delta_k x, x\rangle \geq \min\{t_k, s_k\}\Vert x\Vert^2 \geq 0$.
\end{proof}

\section{Adaptive stepsizes for primal-dual Douglas-Rachford}\label{section-dr-stepsize}
After showing the convergence of the varying preconditioned DR method \eqref{dr-with-delta} and discussing possibilities to apply it to primal-dual problems, we now turn our attention to the problem of how to choose the varying preconditioners $(\Delta_k)_{k\in\N}$ (or varying stepsize sequences $(t_k)_{k\in\N}, (s_k)_{k\in\N}$ in case of the Primal-Dual DR method).
In order to do so, we follow the approach from \cite{adaptive-stepsizes} and choose the preconditioners adaptively during the iterations with the aim to increase the convergence speed.
We start by motivating a heuristic for linear operators $A$, $B$ and $\Delta$ between finite dimensional Hilbert spaces, i.e. matrices.

Let $A$,$B \in \R^{N\times N}$ be monotone matrices and $\Delta\in \R^{N\times N}$ be positive definite. Recall that a matrix $T$ is monotone if
$\langle Tx, x \rangle\geq 0$ for all $x$ (and maximality is implied by linearity).
The preconditioned DR iteration \eqref{dr-with-delta} translates directly into the matrix
$$F_\Delta := \I + J_{\Delta B}(2J_{\Delta A} - \I) - J_{\Delta A},$$
where $\I$ denotes the identity matrix. The iteration $(w^k)_{k\in\N}$ generated by $w^{k+1}=F_{\Delta} w^k$ does not directly converge to a zero of $A+B$, but the sequence $(x_k)_{k\in\N}$ with $x^{k+1} = J_{\Delta A} x^k$ does. Since this section is restricted to single valued mappings, one may insert the substitution $x^{k+1} = J_{\Delta A}w^k$ directly into the iteration $w^{k+1} = F_\Delta w^k$, which results in
$$(\I + \Delta A)x^{k+1} = (\I + \Delta A)x^k + J_{\Delta B}(2x^k - (\I + \Delta A)x^k)-x^k.$$
Multiplication with the resolvent shows the equivalence to
$$x^{k+1} = J_{\Delta A }(\Delta A x^k + J_{\Delta B}(x^k - \Delta A x^k)).$$
Hence, the iteration $x^{k+1} = H_\Delta x^k$ with
\begin{equation}\label{linear-precond-dr}
    H_\Delta := J_{\Delta A}(\Delta A + J_{\Delta B} (\I - \Delta A))
\end{equation}
is equivalent to the fixed point iteration of $F_\Delta$. Furthermore, $H_\Delta$ and $F_\Delta$ are related via
$$F_\Delta = (\I + \Delta A)H_\Delta(\I + \Delta A)^{-1},$$
which implies the similarity of $F_\Delta$ and $H_\Delta$ and the coincidence of their eigenvalues.\\
As noticed in \cite{adaptive-stepsizes} (based on \cite[Theorem 11.2.1]{MatrixComputations}), the asymptotic convergence rate of the linear DR iteration is governed by the spectral radius $\rho (F_\Delta)$ of $F_\Delta$. Furthermore, the similarity of $F_\Delta$ and $H_\Delta$ justifies the minimization of the spectral radius of $H_\Delta$ in order to speed up the convergence speed of $F_\Delta$.

\begin{lemma}\label{lemma-eigenvalues}
Let $A,B\in\R^{N\times N}$ be maximal monotone, $\Delta\in\R^{N\times N}$ be symmetric positive definite, and $H_\Delta$ defined as in \eqref{linear-precond-dr}. Let $\lambda\in\mathbb C$ be an eigenvalue of $H_\Delta$ with the corresponding eigenvector $z\in\mathbb C^N$. Assume that $\lambda \neq 1$ and define
\begin{equation}\label{c-definition}
    c := \frac{\re(\langle Az, z\rangle)}{\Vert z\Vert_{\Delta^{-1}}^2 + \Vert  Az\Vert_\Delta^2},
\end{equation}
where $\re(u)$ denotes the real part of a complex number $u\in\mathbb C$. Then, we have $c\geq 0$ and 
\begin{equation}\label{eigenval-inequality}
    \left\vert \lambda - \frac{1}{2}\right\vert \leq \sqrt{\frac{1}{4} - \frac{c}{1+2c}}\leq \frac{1}{2},
\end{equation}
i.e. it holds $\lambda \in \mathbb B_{\frac{1}{2}}(\frac{1}{2})$.
\end{lemma}
\begin{proof}
Since $\Delta$ is symmetric and positive definite, $\langle x ,y \rangle_{\Delta^{-1}} := \langle \Delta^{-1} x, y \rangle$ is an inner product with induced norm $\Vert x\Vert_{\Delta^{-1}} := \sqrt{\langle x,x\rangle_{\Delta^{-1}}}$. Furthermore, define $\tilde A = \Delta A$ and $\tilde B = \Delta B$. Both $\tilde A$ and $\tilde B$ are maximal monotone with respect to $\langle \cdot,\cdot \rangle_{\Delta^{-1}}$ by the monotonicity of $A$ and $B$ with respect to $\langle\cdot , \cdot \rangle$. Furthermore, $H_\Delta$ from \eqref{linear-precond-dr} becomes
\begin{equation*}
    \begin{split}
        H_\Delta &= (\I + \tilde A)^{-1}(\tilde A + (\I + \tilde B)^{-1}(\I - \tilde A))\\
        &= (\I + \tilde A)^{-1}(\I + \tilde B)^{-1}(\I + \tilde B\tilde A)\\
        &= (\I + \tilde A + \tilde B + \tilde B \tilde A)^{-1}(\I + \tilde B \tilde A).
    \end{split}
\end{equation*}
Therefore, Lemma 2.1 from \cite{adaptive-stepsizes}, applied with $t=1$ to the matrices $\tilde A$ and $\tilde B$ in the Hilbert space endowed with the inner product $\langle\cdot,\cdot\rangle_{\Delta^{-1}}$, states that
$$c = \frac{\re (\langle \tilde A z, z \rangle_{\Delta^{-1}})}{\Vert z\Vert^2_{\Delta^{-1}} + \Vert \tilde A z \Vert^2_{\Delta^{-1}}} = \frac{\re (\langle A z, z\rangle)}{\Vert z \Vert^2_{\Delta^{-1}} + \Vert A z \Vert^2_\Delta},$$
where $c \geq 0$. 
\end{proof}
\begin{figure}
    \centering
    \begin{tabular}[t]{p{0.47\textwidth} p{0.47\textwidth}}
         \subcaptionbox{$\Delta = t \I$}[0.47\textwidth]
        {\includegraphics[width=0.47\textwidth]{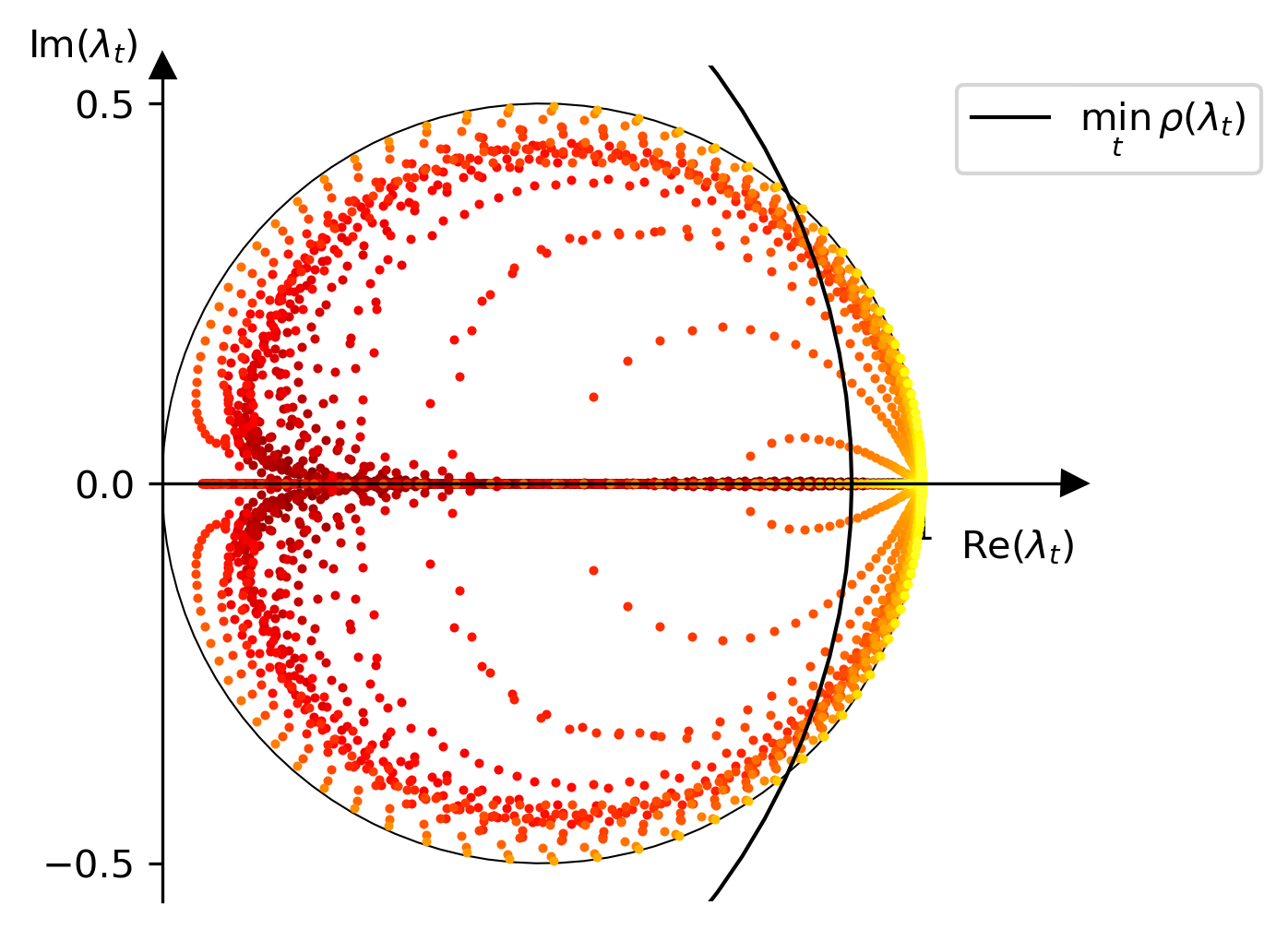}}
        &
        \subcaptionbox{$\Delta = \begin{bmatrix}\tilde t \mathbb I & 0 \\ 0 & \tilde s \mathbb I\end{bmatrix}$}[0.47\textwidth]
        {\includegraphics[width=0.47\textwidth]{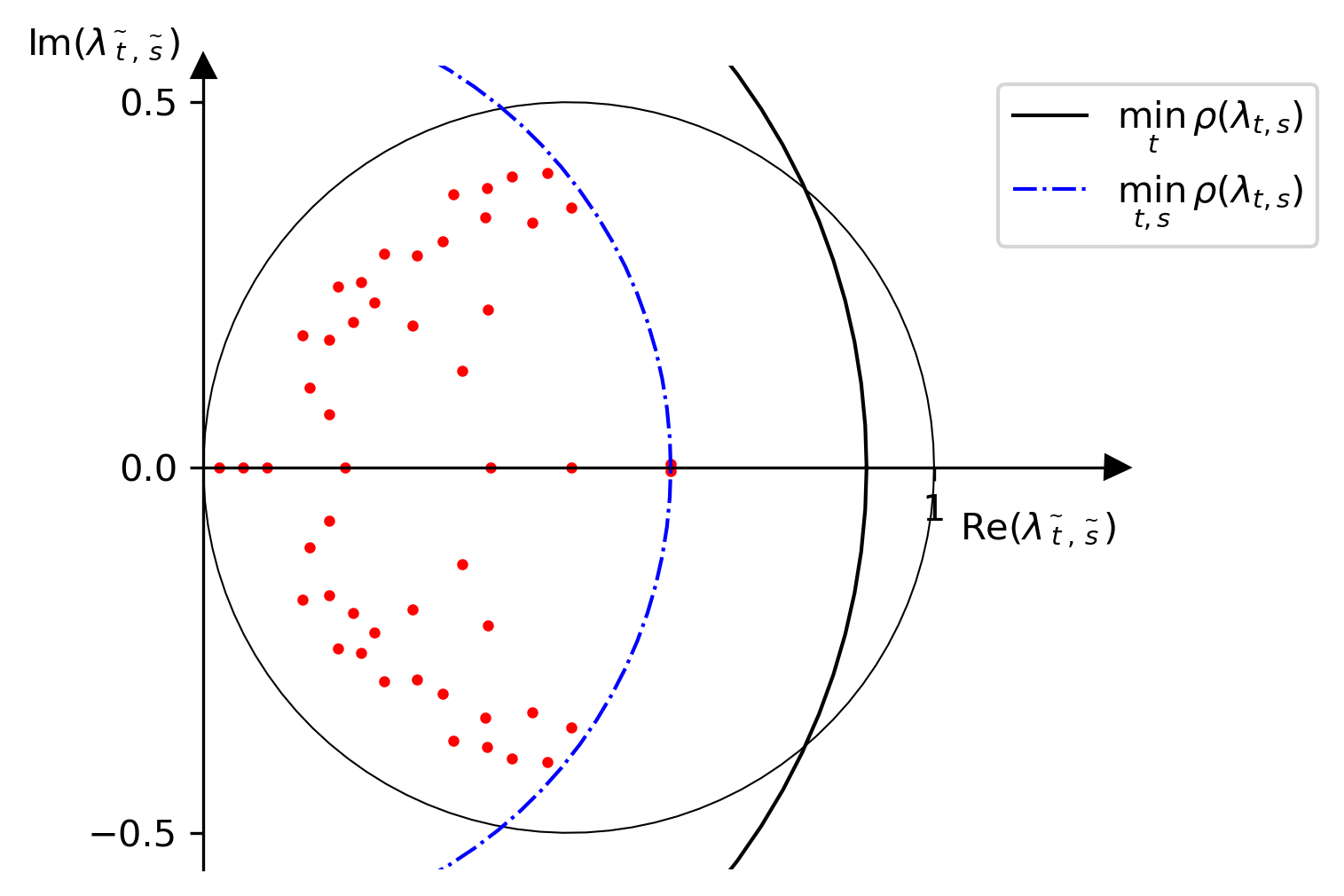}}
        \\
         \subcaptionbox{$\Delta = \begin{bmatrix}t \mathbb I & 0 \\ 0 & \tilde s \mathbb I\end{bmatrix}$}[0.47\textwidth]
        {\includegraphics[width=0.47\textwidth]{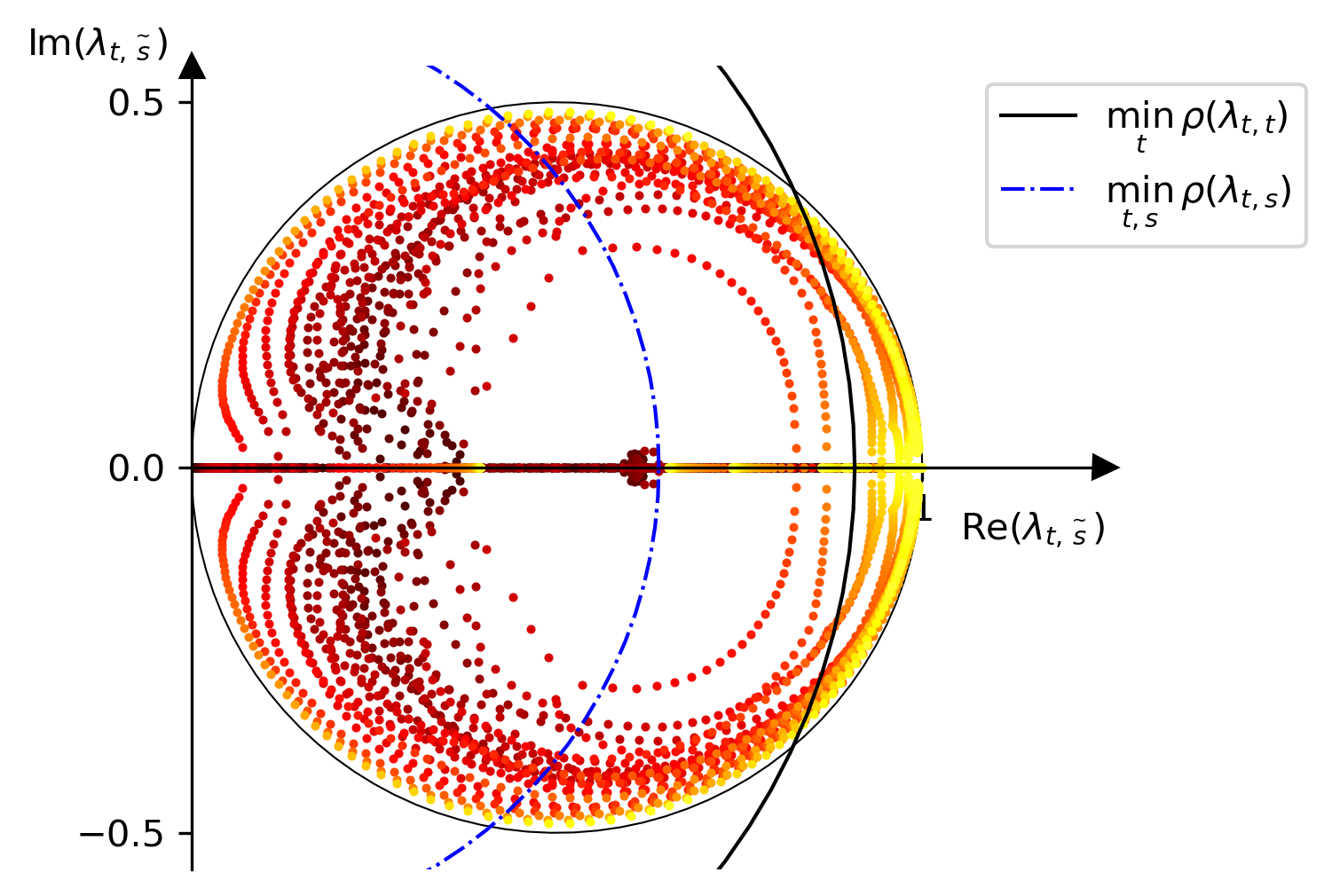}}
        & 
        \subcaptionbox{$\Delta = \begin{bmatrix}\tilde t \mathbb I & 0 \\ 0 & s \mathbb I\end{bmatrix}$}[0.47\textwidth]
        {\includegraphics[width=0.47\textwidth]{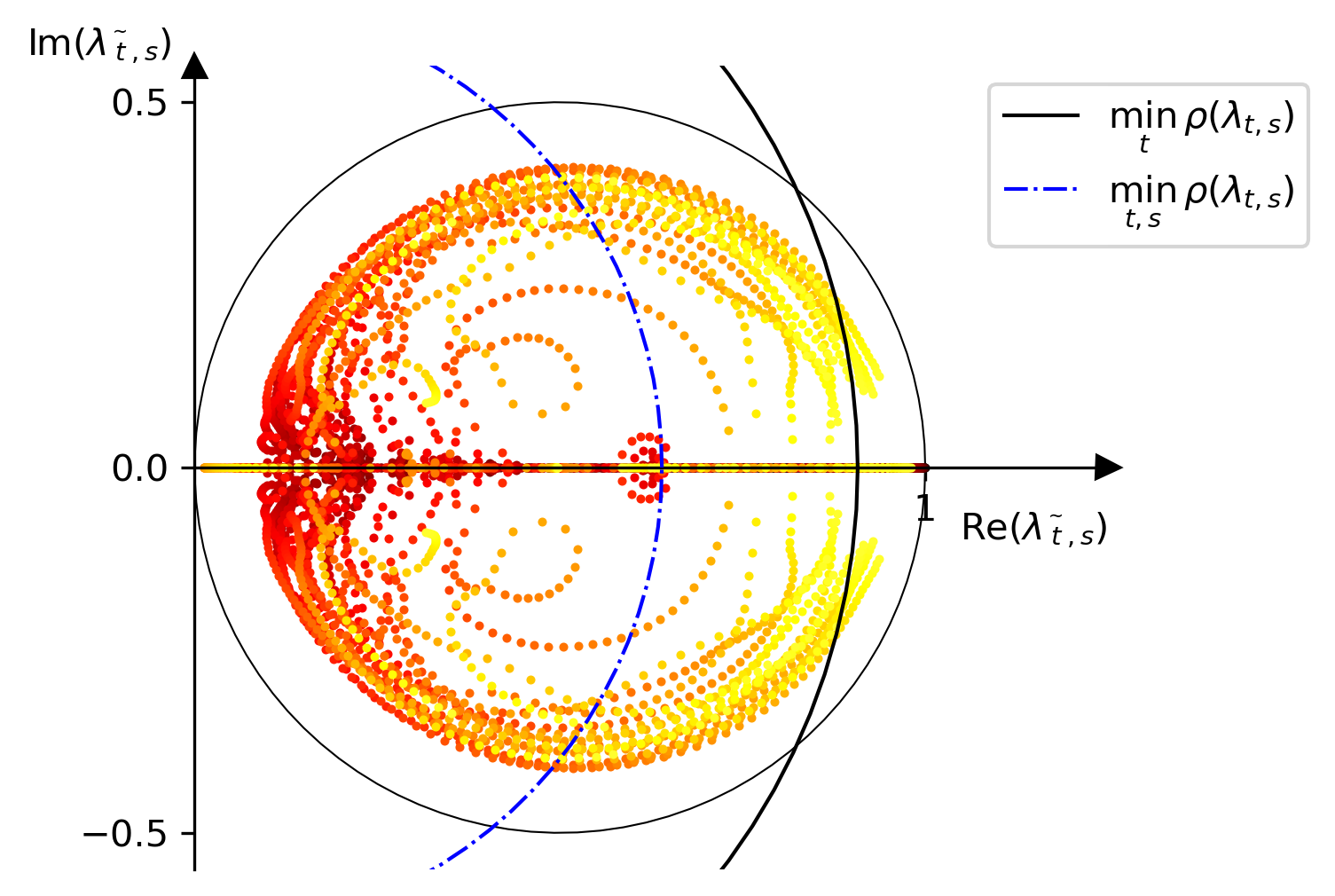}}
    \end{tabular}
    \caption{Illustration of the distribution of eigenvalues of the matrix $H_{\Delta}$ from~\eqref{linear-precond-dr} in the complex plane for two monotone matrices $A,B\in\R^{50\times 50}$ in the form \eqref{eqn:form-of-A-B} with $A_1,A_2\in\R^{25\times 25}$ positive definite and $K\in\R^{25\times 25}$ random with uniformly distributed entries, and different choices of $\Delta\in\R^{50\times 50}$ (see the code available at \url{https://github.com/j-marquardt/adaptive-stepsizes-dr} for details). (a) Eigenvalues of $H_{\Delta}$ in dependence on $t$. The black circle marks the smallest spectral radius of $H_{\Delta}$ that can be obtained in this way (same in the other figures). (b) Eigenvalues of $H_{\Delta}$ where $\tilde t$ and $\tilde s$ are chosen such that the spectral radius of $H_{\Delta}$ is as small as possible. This spectral radius is depicted in blue (same in the other figures). (c) Here we fix $\tilde s$ from figure (b) and show the eigenvalues of $H_{\Delta}$ in dependence on $t$. (d) Same as (c) but with the roles of $t$ and $s$ swapped.}
        \label{ev-circle-fig}
\end{figure}
As proven in Lemma \ref{lemma-eigenvalues} and visualised in Figure \ref{ev-circle-fig}, all eigenvalues of $H_\Delta$ lie in the circle $\mathbb B_{\frac{1}{2}}(\frac{1}{2})$. It seems as if there is no chance to minimize $\rho (H_\Delta)$ explicitly. Nonetheless, Lemma \ref{lemma-eigenvalues} allows to minimize the distance between the absolute value of the largest eigenvalue and $\frac{1}{2}$ by making $\frac{c}{1+2c}$ as large as possible. Since $c \mapsto \frac{c}{1+2c}$ is increasing for $c \geq 0$, the maximization can be accomplished by making
$$c = \frac{\re(\langle Az, z\rangle)}{\Vert z\Vert_{\Delta^{-1}}^2 + \Vert  Az\Vert_\Delta^2}$$
as large as possible and hence its denominator as small as possible.

Now suppose the operators $A,B$ and $\Delta$ to be of the following form

\begin{equation}\label{eqn:form-of-A-B}
        A := \begin{bmatrix}
        A_1 & 0\\ 
        0 & A_2^{-1}
    \end{bmatrix},
    \qquad
    B := \begin{bmatrix}
        0 & K^T\\
        -K & 0
    \end{bmatrix},
    \qquad
    \Delta := \begin{bmatrix}t \I &  0\\ 0 & s\I\end{bmatrix},
    \end{equation}

where $A_1 \in \mathbb{R}^{n\times n}$, $A_2 \in \mathbb{R}^{m\times m}$ are monotone matrices with $A_2$ invertible.

Given $z=(z_1,z_2)$ an eigenvector for $H_\Delta$, with $z_1\in \mathbb{C}^n$ and $z_2\in \mathbb{C}^m$, the value $c$ can be written as
\begin{equation*}
    c = \frac{\re(\langle Az, z\rangle)}{t^{-1}\Vert z_1\Vert^2+ s^{-1}\Vert z_2\Vert^2 + t\Vert  A_1z_1\Vert^2 + s\Vert  A_2^{-1}z_2\Vert^2}.
\end{equation*}
In order to maximize $c$, we minimize the denominator in $t$ and $s$, obtaining as solutions
\[t^*=\frac{\Vert z_1 \Vert}{\Vert A_1z_1 \Vert}, \quad s^*=\frac{\Vert z_2 \Vert}{\Vert A_2^{-1}z_2 \Vert}.\]
Following the work in \cite{adaptive-stepsizes}, this motivates an adaptive choice for the step-sizes $t_k$ and $s_k$ in algorithm \eqref{dr-iteration} as
\[t_{k+1}=\frac{\Vert x^k \Vert}{\Vert \partial f(x^k) \Vert}, \quad s_{k+1}=\frac{\Vert \tilde{y}^k \Vert}{\Vert \partial g( \tilde{y}^k) \Vert}.\]
Where, with a slight abuse of notations, $\partial f(x^k)$ and $\partial g( \tilde{y}^k)$ indicate two elements in the sets $\partial f(x^k)$ and $\partial g( \tilde{y}^k)$ respectively.
From the iterations we have a natural choice for these elements in the subdifferentials: $\frac{p^k-x^k}{t_k}\in \partial f (x^k)$ and $\frac{\tilde q^k-\tilde y^k}{s_k}\in \partial g^{*} (\tilde y^k)$, which justify the choices
\[t_{k+1}=\frac{\Vert x^k \Vert}{\Vert p^k-x^k \Vert}t_k, \quad s_{k+1}=\frac{\Vert \tilde{y}^k \Vert}{\Vert \tilde q^k - \tilde{y}^k \Vert}s_k,\]
however, this choice does not guarantee convergence of the sequences $(t_k)$ and $(s_k)$.
To enforce convergence of the algorithm we choose safeguards $0<a_t<b_t<+\infty$ and $0<a_s<b_s<+\infty$, summable sequences $(\omega_t^k)_{k\in\N}, (\omega_s^k)_{k\in\N}\in (0,1]^\N$ with $\omega_t^0=\omega_s^0=1$ and set initial guesses $t_0$ and $s_0$. The adaptive choice for the step-sizes becomes
\begin{equation}\label{adaptive-stepsizes}
    \begin{aligned}
        t_{k+1} & =\left[(1-\omega_t^k)+\omega_t^k\text{ proj}_{[a_t,b_t]}\left(\frac{\Vert x^k \Vert}{\Vert p^k-x^k \Vert}\right) \right]t_k, \\
        s_{k+1} & =\left[(1-\omega_s^k)+ \omega_s^k\text{ proj}_{[a_s,b_s]}\left(\frac{\Vert \tilde y^k \Vert}{\Vert \tilde q^k- \tilde y^k \Vert}\right)\right]s_k,
    \end{aligned}
\end{equation}
where $\text{proj}_{[a,b]}$, with $a,b\in \mathbb{R}$, denotes the projection onto $[a,b]$. As discussed in \cite{adaptive-stepsizes}, it is possible to show that the sequences generated by such choice satisfy
\begin{equation*}
    \begin{split}
        t_k& \to t > 0, \quad \sum_{k\in\N}\vert t_{k+1} - t_k\vert < \infty,\\
        s_k& \to s > 0, \quad \sum_{k\in\N}\vert s_{k+1} - s_k\vert < \infty.
    \end{split}
\end{equation*}
    This, together with Proposition \ref{final-convergence-result}, guarantees convergence for algorithm \eqref{dr-iteration}.

\begin{remark}
    In the case where $A_1$ and $A_2$ are general maximal monotone operators, the adaptive choices for the stepsizes remain the same.
\end{remark}

\begin{remark}\label{remark-experiment-1}
    The fractions in \eqref{adaptive-stepsizes} may not be defined if $p^k = x^k$ or $\tilde q^k = \tilde y^k$ and one should consider this during implementation. A possible way to do so is by applying another step with the old stepsizes, i.e. $t_{k+1}=t_k$ and $s_{k+1} = s_k$ or by setting $\frac{\Vert x^k\Vert}{\Vert p^k - x^k\Vert} = b_t$ and $\frac{\Vert \tilde y^k\Vert}{\Vert \tilde q^k - \tilde y^k\Vert} = b_s$ since these are the values of the sourrounding projections for $\Vert p^k - x^k\Vert \approx 0$, $\Vert x^k \Vert \neq 0$ and $\Vert \tilde q^k - \tilde y^k\Vert \approx 0$, $\Vert \tilde y^k \Vert \neq 0$.
    
    The occurrence of this error is not only pure theoretical, but i.e. occurs for the Experiment in Section \ref{experiment-tv}. 
\end{remark}

\begin{remark}\label{remark-experiment-2}
    The convergence of the sequences \eqref{adaptive-stepsizes} with summable distance to the limit is guaranteed by \cite[Lemma 4.1]{adaptive-stepsizes}. Nevertheless, the limits $t$ and $s$ may be so large that the sequences cause overflow errors. In order to prevent these errors we bound the sequences $(t_k)_{k\in\N}$ and $(s_k)_{k\in\N}$ by a constant from above (this will be needed in the experiment in Section \ref{experiment-tv}).
\end{remark}

\section{Experiments}\label{section-experiments}

\subsection{Constant $t$ vs adaptive $t_k$ vs adaptive $t_k$ and $s_k$}

In a first experiment we consider the problem of $\ell^1$-regularized least absolute deviations
\begin{align*}
    \min_x \Big\{ F(x) := \|Ax-b\|_1 + \lambda \|x\|_1\Big\},
\end{align*}
where $A\in\R^{m\times n}, b\in \R^m$ and $\lambda > 0$. We solve this problem using the DR iteration \eqref{dr-iteration} with $f(x) = \lambda  \Vert x\Vert_1$, $g(y) =  \Vert y - b \Vert_1$, $K=A$ and therefore
\begin{equation*}
    \begin{split}
    \prox_{t f} (x) &= \text{sign} (x) \odot \max(\vert x\vert - t \lambda, 0),\\
    \prox_{s g^*}(y) &= y - sb - \text{sign}(y-sb) \odot \max(\vert y-sb\vert - 1, 0),
    \end{split}
\end{equation*}
where $\text{sign}(\cdot)$, $\vert\cdot\vert$ and $\max(\cdot)$ are applied componentwise.

In this experiment we compare the primal-dual DR method with a constant (well chosen) stepsize, with a single adaptive stepsize $t_k$ from \cite{adaptive-stepsizes} and two adaptive stepsizes $t_k$ and $s_k$ according to the rule \eqref{adaptive-stepsizes}.

The results for the primal-dual DR method for these stepsize rules is shown in Figure \ref{fig:lad-comparison}. The constant stepsize and the $t$-adaptive stepsize show similar convergence behaviours. This is no surprise, since the stepsize sequence of the $t$-adaptive method converges to a value close to the constant stepsize $t=1.1$ quite fast. The $(t,s)$-adaptive choice for stepsizes varies more until it settles around a stable value. This behavior is related to the decay of the objective function's values. The fast decay of the $(t,s)$-adaptive method starts and outruns the decay of the other two methods once the initial variations in the stepsize sequences are overcome. 

The behavior shown in Figure \ref{fig:lad-comparison} is typical and can be observed for the majority of initializations for $A$, $b$ and $\lambda$. The larger oscillations at the beginning of the iterations appear to be common for the $(t,s)$-adaptive stepsize rule and allow both $t$ and $s$ to assume better values.

\begin{figure}[ht]
\centering
	\begin{subfigure}[t]{0.49\textwidth}
        \vskip 0pt
		\centering
		\includegraphics[width=1\linewidth]{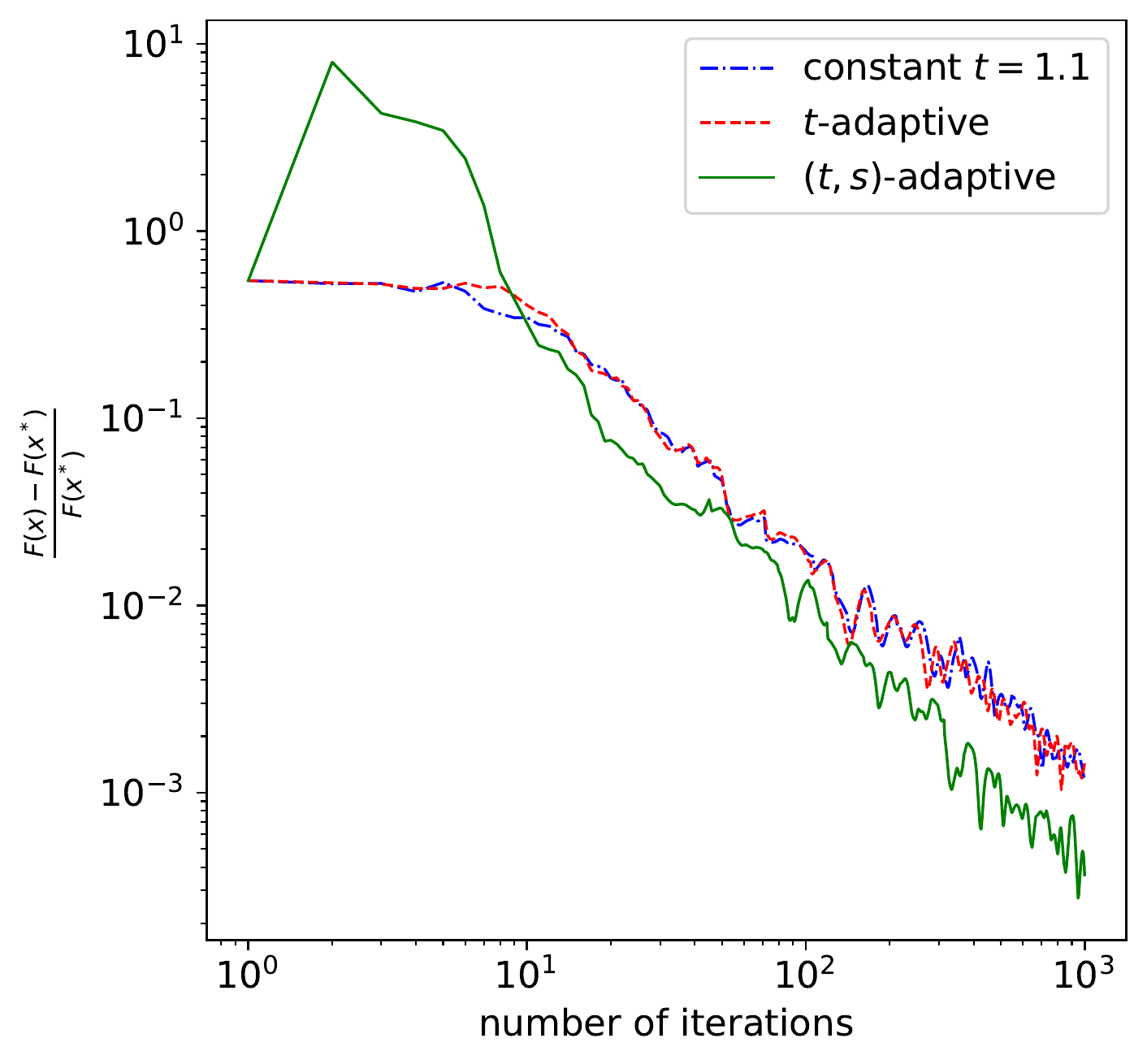}
		\caption{Decay of objective values\\\,}
	\end{subfigure}
	\hfill
	\begin{subfigure}[t]{0.49\textwidth}
        \vskip 0pt
		\centering
		\includegraphics[width=1\linewidth]{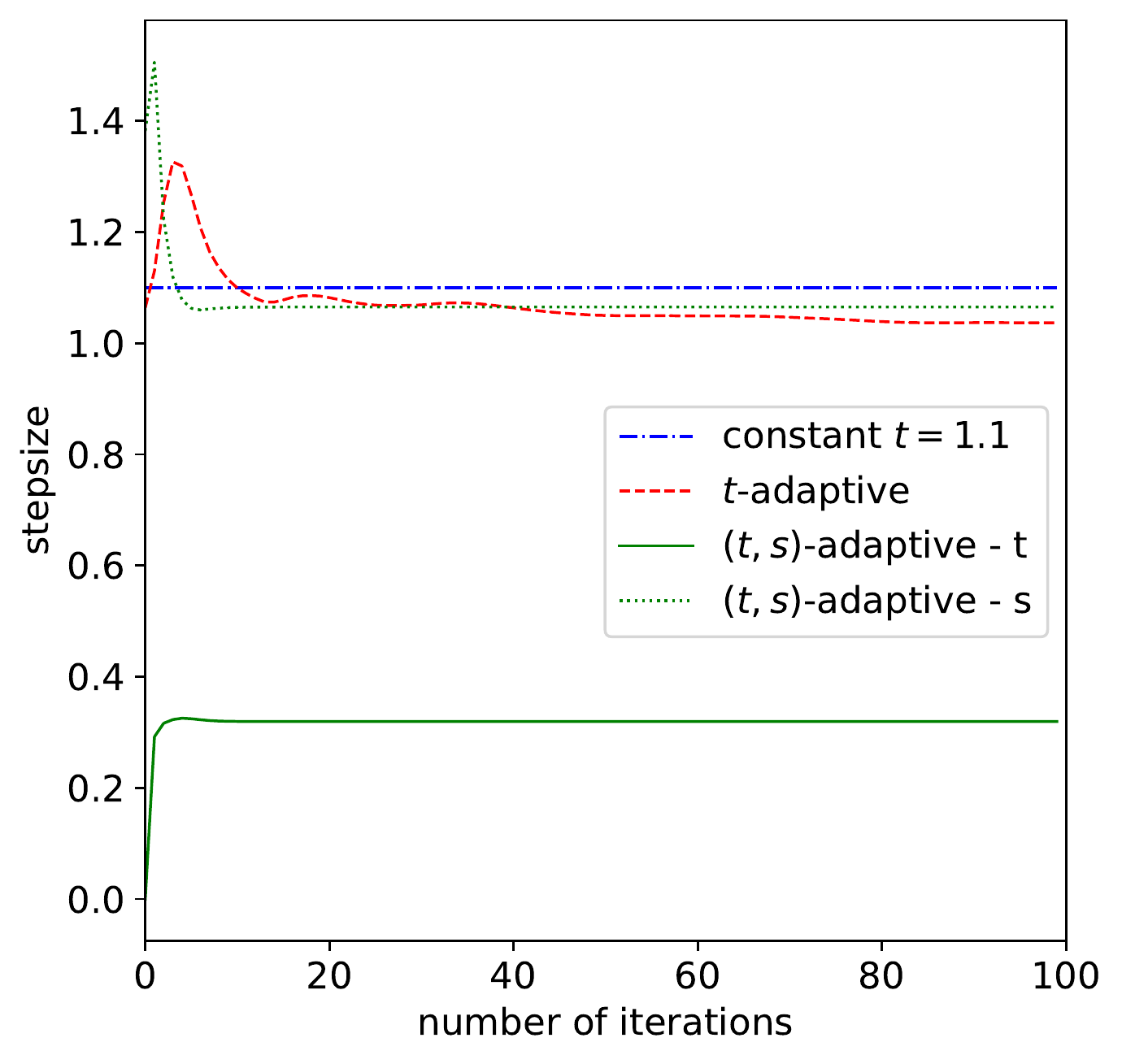}
		\caption{Development $t$-adaptive and $(t,s)$-adaptive stepsize sequences for the first $100$ timesteps.}
	\end{subfigure}
\caption{Experiment in Section 5.1 with three different choices for stepsizes sequences: only one stepsize $t=1.1$ as a good choice for a constant stepsize sequence; one single adaptive stepsize sequence $t_k=s_k$ as described in \cite{adaptive-stepsizes}; two adaptive stepsize sequences $t_k$ and $s_k$ as in \eqref{adaptive-stepsizes} with $\omega^k_t = \omega^k_s= 2^{-k}$, $a_t=a_s = 10^{-4}$, $b_t=b_s=10^4$ and $t_0=s_0=1$. The matrix $A\in\R^{200\times 100}$, $b\in \R^{200}$ have been chosen randomly, the parameter $\lambda=1$ was chosen fixed (see the code available at \url{https://github.com/j-marquardt/adaptive-stepsizes-dr} for details).}
\label{fig:lad-comparison}
\end{figure}

\subsection{Stepsizes adapt to problem parameters}\label{experiment-tv}
\begin{figure}[ht!]
\centering
	\begin{subfigure}{0.49\textwidth}
	\centering
	\includegraphics[width=1\textwidth]{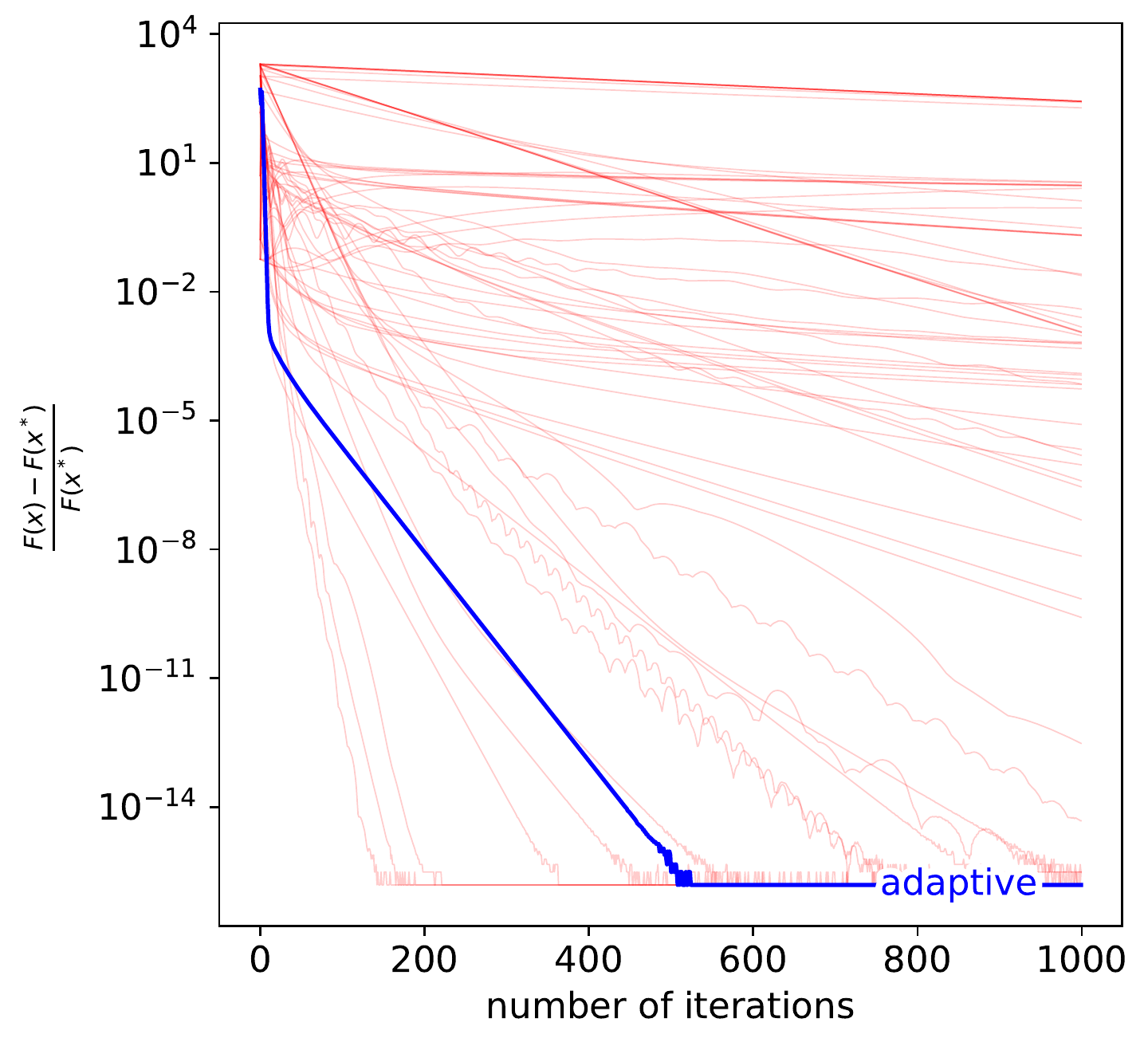}
	\caption{$\lambda=0.01,t\approx 1.55, s\approx 0.03$}
	\end{subfigure}
	\hfill
	\begin{subfigure}{0.49\textwidth}
		\centering
		\includegraphics[width=1\linewidth]{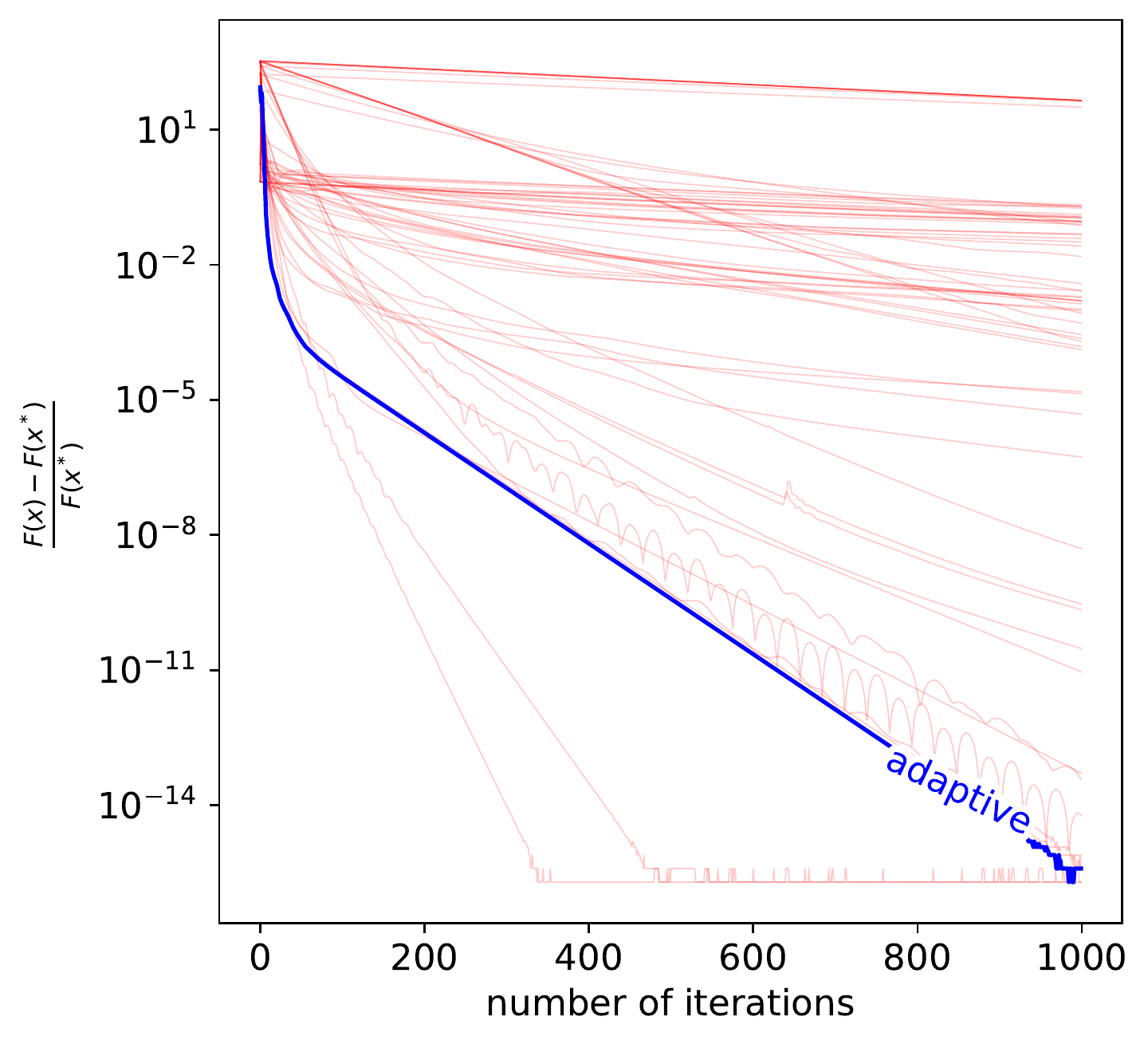}
		\caption{$\lambda=0.1,t\approx 1.38, s\approx 0.41$}
	\end{subfigure}
    \begin{subfigure}{0.49\textwidth}
		\centering
		\includegraphics[width=1\linewidth]{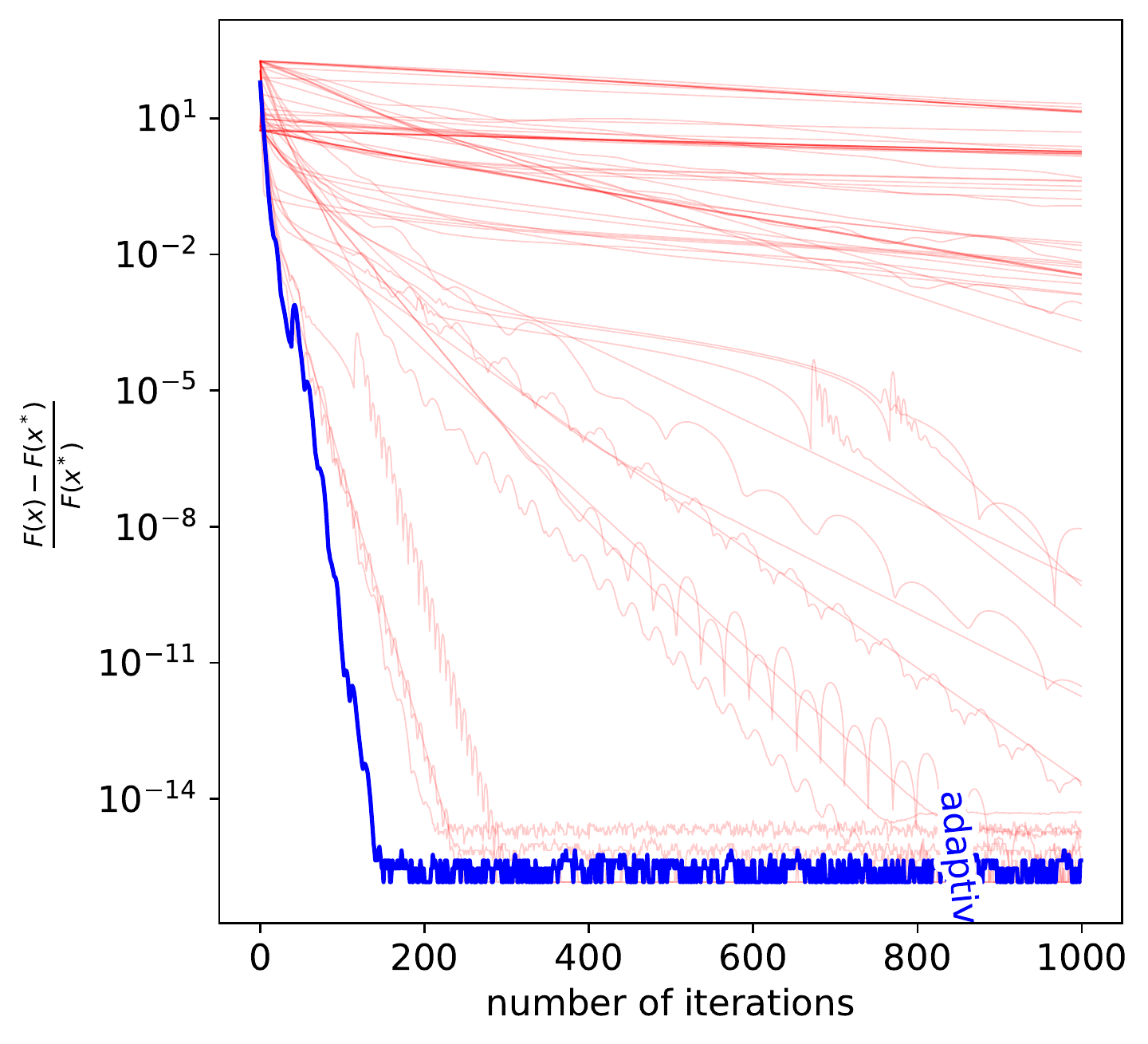}
		\caption{$\lambda=1,t\approx1.303, s\approx 5.38$}
	\end{subfigure}
    \begin{subfigure}{0.49\textwidth}
		\centering
		\includegraphics[width=1\linewidth]{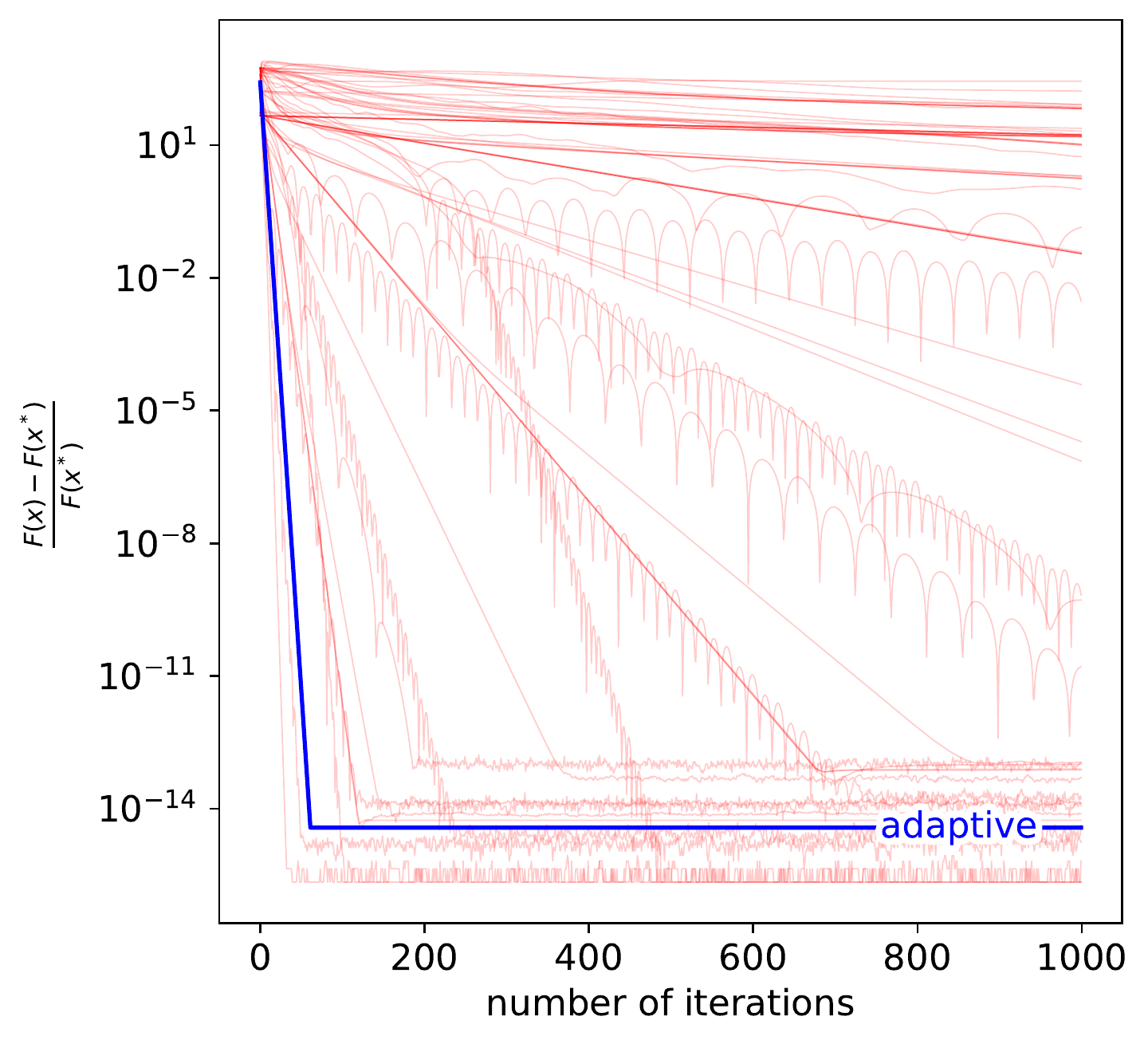}
		\caption{$\lambda=10,t\approx1.07, s= 10^{4}$}
	\end{subfigure}
 
	\caption{Results for the second experiment described in Section~\ref{experiment-tv}. Parameters $\omega_t^k = \omega_s^k = 2^{-k}, a_t = a_s = 10^{-4}, b_t = b_s = 10^4$. An upper bound for $(t_k)_{k\in\N}$ and $(s_k)_{k\in\N}$ as proposed in Remark \ref{remark-experiment-2} has been chosen as $10^4$ (see the code available at \url{https://github.com/j-marquardt/adaptive-stepsizes-dr} for details).}
	\label{fig:comparisons}
\end{figure}
For this second experiment, we consider the task of denoising. Given a noisy signal $x^\delta$ we want to generate a reconstruction using some prior knowledge. A common way to do so is to use the total variation (TV) denoising \cite{Rudin1992,Getreuer2012}. The aim is to solve the following problem
\[\min_x \Big\{F(x):= \frac{1}{2}\|x-x^\delta\|_2^2+\lambda \|Dx\|_1\Big\}\]
where $D$ is a finite differences operator. The first term in the function $F$ is called fidelity term and assures that the solution is in some sense "close" to the data $x^\delta$ (in this case, close in 2-norm), while the second term is the total variation term, our prior, and assures some penalization on the gradient of the solution and promote sparsity for it. We solve this problem using the primal-dual DR iteration \eqref{dr-iteration} with $f(x) = \frac{1}{2}\Vert x-x^{\delta}\Vert_2^2$, $g(y) = \lambda \Vert y \Vert_1$ and $K=D$. While performing such reconstructions, the role of the paramenter $\lambda$ is crucial. This parameter can depend on the type of signals we are considering and on the level of noise. Depending on $\lambda$, $x^{\delta}$ and $D$, the choice for good stepsizes could vary significantly. We want to show that, for a given $\lambda$ our proposed adaptive primal-dual DR method recovers good stepsizes. We perform four experiments, with different values for $\lambda$ and report the values for $t$ and $s$ at convergence. All four experiments are performed with the choice $t_0=1,s_0=1$ and $\omega_t^k=\omega_s^k=2^{-k}$ and confronted with iterations of primal-dual DR algorithms with constant choice of stepsizes. In particular, we choose $100$ combinations of $t$ and $s$ from $10$ choices in a log scale from $10^{-3}$ to $10^3$. Figure~\ref{fig:comparisons} shows the results of these runs with constant stepsizes in red while the result of the runs with the adaptive choice~\eqref{adaptive-stepsizes} is shown in blue. One observes that the performance of the method with constant stepsizes varies significantly (from almost no progress in 1000 iterations to convergence in a dozen of steps). The adaptive stepsize, however, is always among the fastest runs. Moreover, we see that the method~\eqref{adaptive-stepsizes} does indeed adapt the dual stepsize $s$ to the parameter $\lambda$ (the $t$ also changes slightly).

During the runs of the $(t,s)$-adaptive DR method, Remarks \ref{remark-experiment-1} and \ref{remark-experiment-2} became crucial. Especially the dual stepsize sequence $(s_k)_{k\in\N}$ triggered both the safeguard which prevents errors in case of $\tilde q^k = \tilde y^k$ and the necessity to bound $(s_k)_{k\in\N}$ by i.e. $10^4$ to prevent overflow errors in case of $\lambda = 10$. Both of them can be explained by the structure of the dual $\prox$-operator, i.e. $q= \prox_{s (\lambda \Vert\cdot\Vert_1)^*}(y) = y - \text{sign}(y) \odot \max(\vert y\vert - \lambda, 0)$, which is independent of the stepsize $s$ and equal to the identity for small arguments $y$ (but note that the dual stepsize $s$ still plays a role in the resolvent of $B$).

\section*{Acknowledgments} This work has received funding from the European Union’s Framework Programme for Research and Innovation Horizon 2020 (2014-2020) under the Marie Sk\l odowska-Curie Grant Agreement No. 861137.

\bibliographystyle{tfq}
\bibliography{references}

\end{document}